\documentclass[letterpaper,10pt]{article}
\usepackage{amsmath,amsfonts,amsthm,amssymb,graphicx,paralist,mathtools,algorithm,multirow}
\usepackage{fullpage,algpseudocode}
\usepackage[dvipsnames]{xcolor}

\DeclareMathOperator{\dom}{dom}

\newcommand{\R}{{\mathbf R}}

\newcommand{\bp}{\begin{proof}[\textbf{Solution}]}
\newcommand{\ep}{\end{proof}}

\newcommand{\p}{\mathcal{P}}

\newcommand{\eqdef}{:=}

\newcommand{\kk}{^{(k)}}
\newcommand{\kpo}{^{(k+1)}}

\newcommand{\Sscr}{\mathcal{S}}
\newcommand{\Ascr}{\mathcal{A}}

\DeclareMathOperator*{\minim}{minimize}
\newcommand{\minimize}[1]{{\displaystyle\minim_{#1}}}
\DeclareMathOperator*{\subject}{subject\,\, to}

\newtheorem{theorem}{Theorem}
\newtheorem{lemma}[theorem]{Lemma}

\newtheorem{example}[theorem]{Example}
\newtheorem{remark}[theorem]{Remark}
\newtheorem{assumption}[theorem]{Assumption}

\title{A Flexible ADMM Algorithm for Big Data Applications}
\author{Daniel P. Robinson and Rachael E. H. Tappenden}
\date{}
\begin{document}
\maketitle\vspace{-5mm}
\begin{abstract}
We present a flexible Alternating Direction Method of Multipliers
(F-ADMM) algorithm for solving optimization problems involving a
strongly convex objective function that is separable into $n \geq 2$
blocks, subject to (non-separable) linear equality constraints. The
F-ADMM algorithm uses a \emph{Gauss-Seidel} scheme to update blocks of
variables, and a regularization term is added to each of the
subproblems arising within F-ADMM. We prove, under common assumptions,
that F-ADMM is globally convergent.

We also present a special case of F-ADMM that is \emph{partially
  parallelizable}, which makes it attractive in a big
data setting. In particular, we partition the data into groups, so
that each group consists of multiple blocks of variables. By applying
F-ADMM to this partitioning of the data, and using a specific
regularization matrix, we obtain a hybrid ADMM (H-ADMM) algorithm: the
grouped data is updated in a Gauss-Seidel fashion, and the blocks
within each group are updated in a Jacobi manner. Convergence of H-ADMM follows
directly from the convergence properties of F-ADMM. Also, a special case of
H-ADMM can be applied to functions that are convex, rather than
strongly convex. We present numerical experiments to demonstrate the
practical advantages of this algorithm.


\end{abstract}

\section{Introduction}
In this work we study the optimization problem
\begin{subequations}
\label{Problem}
\begin{eqnarray}
\label{Problem_objective}
     &\minimize{x_1,\dots, x_n}&  \sum_{i=1}^n f_i(x_i)\\
\label{Problem_constraints}
&\subject& \sum_{i=1}^n A_ix_i =b,
\end{eqnarray}
\end{subequations}
where, for each $i= 1,\dots,n$, the function $f_i: \R^{N_i}\to \R \cup \{\infty\}$ is strongly convex, closed, and extended real valued, and the vector $b\in \R^m$ and matrix $A_i \in \R^{m \times N_i} $ represent problem data.
Note that the objective function \eqref{Problem_objective} is
separable in the decision vectors $x_1,\dots,x_n$, but that the linear
constraint \eqref{Problem_constraints} links them together, which
makes problem \eqref{Problem} non-separable overall. 


We can think of the decision vectors $\{x_i\}$ as
``blocks" of a single decision vector $x \in \R^N$, where
$N=\sum_{i=1}^n N_i$. This can be achieved by partitioning the $N
\times N$ identity matrix $I$ column-wise into $n$ submatrices $\{U_i
\in \R^{N\times N_i}\}_{i=1}^n$, so that
$I=[U_1,\dots,U_n]$, and then setting $x = \sum_{i=1}^n U_i x_i$. That
is, $x$ is the vector formed by stacking the vectors $\{x_i\}_{i=1}^n$ on top of
each other. It is easy to see that $x_i =
U_i^T x \in \R^{N_i}$, and that if we let $A \eqdef \sum_{i=1}^n A_i
U_i^T \in \R^{m\times N},$ then \eqref{Problem_constraints} is
equivalent to $Ax = b$. Note that $A_i = A U_i$ for
$i=1,2,\dots,n$, and that we can write $A = [A_1,\dots,A_n]$. If we
now let $f(x) \eqdef \sum_{i=1}^n f_i (x_i)$, problem \eqref{Problem}
is equivalent to
\begin{subequations}
\label{Problem_compact}
\begin{eqnarray}
\label{Problem_compact_objective}
     &\minimize{x \in \R^N}&  f(x)\\
\label{Problem_compact_constraints}
&\subject&  Ax =b.
\end{eqnarray}
\end{subequations}
Although problems~\eqref{Problem} and~\eqref{Problem_compact} are mathematically equivalent, it is important to note that the best algorithms for solving them take advantage of the block structure that is made explicit in formulation~\eqref{Problem}.

\subsection{Relevant Previous Work}

Many popular algorithms for solving~\eqref{Problem} (equivalently, for solving~\eqref{Problem_compact}) are based on the Augmented Lagrangian function. In the remainder of this section, we describe several such algorithms that are closely related to our proposed framework.

\subsubsection*{The Augmented Lagrangian Method of Multipliers (ALMM)}


The ALMM (e.g., see~\cite{Bertsekas96}) is based on the augmented Lagrangian function
\begin{equation}\label{Lagrangian}
  \mathcal{L}_\rho(x; y) \eqdef f(x) - \langle y, Ax - b\rangle + \frac{\rho}{2}\|Ax - b\|_2^2,
\end{equation}
where $\rho>0$ is called the penalty parameter, $y\in \R^m$ is a dual
vector that estimates a Lagrange multiplier vector, and
$\langle p,q \rangle = p^T \!q$ is the standard inner product in $\R^n$. The most basic
variant of ALMM (see Algorithm~\ref{ALMM}), involves two key steps
during each iteration.  First, for a fixed dual estimate, the
augmented Lagrangian~\eqref{Lagrangian} is minimized with respect to
the primal vector $x$.  Second, using the minimizer computed in the first
step, a simple update is made to the dual vector that is equivalent to
a dual ascent step for maximizing an associated dual function.  In
practice, computing the minimizer in the first step is the
computational bottleneck.  This is especially true for large-scale
problems that arise in big data applications, and therefore extensive
research has focused on reducing its cost (e.g., decomposition
techniques \cite{Mulvey92,Ruszczynski95,Tappenden14}).

\begin{algorithm}[H]
\caption{A basic variant of ALMM for solving problem~\eqref{Problem_compact}.}\label{ALMM}
  \begin{algorithmic}[1]
    \State \textbf{Initialization:} $y^{(0)} \in \R^m$, iteration counter $k=0$, and penalty parameter $\rho > 0$.
    \While {the stopping condition has not been met}
    \State Update the primal variables by minimizing the augmented Lagrangian:
\begin{equation}
     \label{ALMM_minimize}
     \displaystyle x^{(k+1)} \leftarrow \arg \min_{x}\, \mathcal{L}_\rho(x;y^{(k)})
\end{equation}
\State Update the dual variables:
\begin{equation*}
     y^{(k+1)} \leftarrow y^{(k)} - \rho(Ax^{(k+1)}-b)
\end{equation*}
\State Set $k\gets k+1$.
\EndWhile
  \end{algorithmic}
\end{algorithm}

Although sophisticated variants of ALMM are successfully used in many important application areas (e.g., optimal control in natural gas networks~\cite{zavalastochastic}), generally they are unable to directly take advantage of the block separability described in formulation~\eqref{Problem}, when it exists.  Nonetheless, ALMM serves as the basis for many related and powerful methods, as we now discuss.

\subsubsection*{The Alternating Direction Method of Multipliers (ADMM)}\label{S_ADMM}


The ADMM has been a widely used algorithm for solving problems of the form \eqref{Problem} when $n =2$, for convex functions. Global convergence of ADMM was established in the early 1990's by Eckstein and Bertsekas~\cite{Eckstein92} while studying the algorithm as a particular instance of a Douglas-Rachford splitting method.  This relationship allowed them to use monotone operator theory to obtain their global convergence guarantees.
(An introduction to ADMM and its convergence theory can be found in the tutorial style paper by Eckstein \cite{Eckstein12}. See also \cite{Boyd10}.)
Pseudocode for ADMM when $n = 2$ is given below as Algorithm~\ref{ADMM}.
\begin{algorithm}[H]
\caption{ADMM for solving problem~\eqref{Problem} when $n = 2$.}\label{ADMM}
  \begin{algorithmic}[1]
    \State \textbf{Initialization:} $x^{(0)}\in\R^N$, $y^{(0)} \in \R^m$,
    iteration counter $k=0$, and penalty parameter $\rho > 0$.
    \While {the stopping condition has not been met}
    \State Update the primal variables in a Gauss-Seidel fashion:
\begin{subequations}
\begin{eqnarray}
     \label{ADMM_block1}
     \displaystyle x_1^{(k+1)} &\leftarrow& \arg \min_{x} \,\mathcal{L}_\rho(x,x_2^{(k)};y^{(k)})\\
     \label{ADMM_block2}
     \displaystyle x_2^{(k+1)} &\leftarrow& \arg \min_{x} \,\mathcal{L}_\rho(x_1^{(k+1)},x;y^{(k)})
\end{eqnarray}
\State Update the dual variables:
\begin{equation*}
     y^{(k+1)} \leftarrow y^{(k)} - \rho(Ax^{(k+1)}-b)
\end{equation*}
\end{subequations}
\State Set $k\gets k+1$.
\EndWhile
  \end{algorithmic}
\end{algorithm}

In words, ADMM works as follows. At iteration $k$, for a fixed
multiplier $y\kk$ and fixed block $x_2\kk$, 
 the new point $x_1\kpo$ is defined as the minimizer (for simplicity, we assume throughout that this minimizer exists and that it is unique) of the augmented Lagrangian with respect to the first block of variables $x_1$.  Then, in a similar fashion, the first (updated) block $x_1\kpo$ is fixed, and the augmented Lagrangian is minimized with respect to the second block of variables $x_2$ to obtain $x_2\kpo$. Finally, the dual variables are updated in the same manner as for the basic ALMM (see Algorithm~\ref{ALMM}),
 and the process is repeated. Notice that a key feature of ADMM is that the blocks of variables $x_1$ and $x_2$ are updated in a \emph{Gauss-Seidel} fashion, i.e., the \emph{updated} values for the first block of variables are used to define the subproblem used to obtain the updated values for the second block of variables. The motivation for the design of ADMM is that each subproblem (see~\eqref{ADMM_block1} and \eqref{ADMM_block2}) should be substantially easier to solve than the subproblem (see~\eqref{ALMM_minimize}) used by ALMM.  For many important applications, this is indeed the case.

The interest in ADMM has exploded in recent years because of applications in signal and image processing, compressed sensing \cite{Yang11}, matrix completion \cite{Yuan09}, distributed optimization and statistical and machine learning \cite{Boyd10}, and quadratic and linear programming~\cite{Boley13}. Convergence of ADMM has even been studied for specific instances of nonconvex functions, namely consensus and sharing problems \cite{Hong14}.

A natural question to ask is whether ADMM converges when there are more than two blocks, i.e., when $n \geq 3$.
The authors in~\cite{Chen13} show via a counterexample that ADMM is not necessarily convergent if $n=3$. However, they also show that if $n=3$ and at least two of the matrices that define the linking constraint \eqref{Problem_constraints} are orthogonal, then ADMM will converge. In a different paper~\cite{Caia14}, the authors show that ADMM will converge when $n = 3$ if at least one of the functions $f_i$ in \eqref{Problem_objective} is strongly convex.

Other works have considered the more general case of $n \geq 2$. For example, an ADMM-type algorithm for $n\geq2$ blocks is introduced in \cite{Wang14}, where during each iteration a randomly selected subset of blocks is updated in parallel. The method incorporates a ``backward step" on the dual update to ensure convergence. Hong and Luo~\cite{Hong12} present a convergence proof for the $n$ block ADMM when the functions are convex, but under many assumptions that are difficult to verify in practice. Work in \cite{Han12} shows that ADMM is convergent in the $n$ block case when the functions $f_i$ for $i = 1,\dots ,n$ are strongly convex.

\subsubsection*{The Generalized ADMM (G-ADMM)}

Deng and Yin~\cite{Deng12} introduced G-ADMM, which is a variant of ADMM for solving problems of the form \eqref{Problem} when $n=2$ and the functions $f_i$ are convex. They proposed the addition of a (general) regularization term to the augmented Lagrangian function during the minimization subproblem within ADMM and the addition of a relaxation parameter $\gamma$ to the dual variable update. Their motivation for the inclusion of a regularization term was twofold.  First, for certain applications, a careful choice of that regularizer lead to subproblems that were significantly easier to solve.  Second, the regularization stabilized the iterates, which has theoretical and numerical advantages.

Their method is stated below as Algorithm~\ref{GADMM}.  It uses, for
any symmetric positive-definite matrix $M$ and vector $z$, the
ellipsoidal norm
\begin{equation}
  \|z\|_M^2 \eqdef z^T M z.
\end{equation}

\begin{algorithm}[H]
\caption{G-ADMM for solving problem~\eqref{Problem} when $n=2$.}\label{GADMM}
  \begin{algorithmic}[1]
    \State \textbf{Initialization:} $x_1^{(0)}\in\R^{N_1}$, $x_2^{(0)}\in\R^{N_2}$, $y^{(0)} \in \R^m$, iteration counter $k=0$, parameters $\rho>0$ and $\gamma\in(0,2)$, and regularization matrices $P_1\in\R^{N_1\times N_1}$ and $P_2\in\R^{N_2\times N_2}$.
    \While {the stopping condition has not been met}
    \State Update the primal variables in a Gauss-Seidel fashion:
\begin{subequations} \label{GADMM_primal}
\begin{eqnarray}
     \label{GADMM_block1}
     \displaystyle x_1^{(k+1)} &\leftarrow& \arg \min_{x} \,\mathcal{L}_\rho(x,x_2^{(k)};y^{(k)}) + \tfrac12\|x - x_1^{(k)}\|_{P_1}^2\\
     \label{GADMM_block2}
     \displaystyle x_2^{(k+1)} &\leftarrow& \arg \min_{x} \,\mathcal{L}_\rho(x_1^{(k+1)},x;y^{(k)})  + \tfrac12\|x - x_2^{(k)}\|_{P_2}^2
\end{eqnarray}
\State Update the dual variables:
\begin{equation*}
     y^{(k+1)} \leftarrow y^{(k)} - \gamma\rho(Ax^{(k+1)}-b)
\end{equation*}
\end{subequations}
\State Set $k\gets k+1$.
\EndWhile
  \end{algorithmic}
\end{algorithm}

The authors prove~\cite{Deng12} that Algorithm \ref{GADMM} converges
to a solution from an arbitrary starting point as long as the
regularization matrices $P_1$ and $P_2$ in \eqref{GADMM_block1} and
\eqref{GADMM_block2}
satisfy certain properties. 
We stress that the convergence analysis for G-ADMM only applies to the $n = 2$ case.


\subsubsection*{The Jacobi ADMM (J-ADMM)}

Deng et al.~\cite{Deng14} have extended the ideas first presented in G-ADMM~\cite{Deng12}.  Their new J-ADMM strategy (stated below as Algorithm~\ref{JADMM}) may be used to solve problem~\eqref{Problem} in the general case of $n\geq 2$ blocks.
Note that~\eqref{JADMMupdate} is equivalent to the update
\begin{equation*}
  \displaystyle x_i^{(k+1)} \leftarrow \arg \min_{x_i} \,\mathcal{L}_\rho(x_1\kk,\dots,x_{i-1}\kk,x_i,x_{i+1}\kk,\dots,x_n\kk;y^{(k)})  + \tfrac12\|x_i - x_i^{(k)}\|_{P_i}^2,
\end{equation*}
(where $P_i\in \R^{N_i \times N_i}$ is a regularization matrix) which we state in order to highlight the relationship of their method to the previous ones. We also comment that the form of the update used in~\eqref{JADMMupdate} motivates why their algorithm is of the proximal type.

\begin{algorithm}[H]
\caption{J-ADMM for solving problem~\eqref{Problem} for $n \geq 2$.}\label{JADMM}
  \begin{algorithmic}[1]
    \State \textbf{Initialize:}  $x^{(0)} \in \R^N$, $y^{(0)} \in \R^m$, iteration counter $k=0$, parameters $\rho>0$ and $\gamma \in (0,2)$, and regularization matrices $P_i\in\R^{N_i \times N_i}$ for $i = 1,\dots,n$. \begin{subequations}
    \While {stopping condition has not been met}
    \For{$i = 1,\dots,n$ (in parallel)}
\begin{eqnarray}\label{JADMMupdate}
      \displaystyle x_i^{(k+1)} \leftarrow \arg \min_{x_i} \Big{\{}f_i(x_i) +  \frac{\rho}2\|A_ix_i + \sum_{j \neq i}^n A_jx_j^{(k)} - b - \frac{y\kk}{\rho}\|_2^2  + \frac{1}{2}\|x_i-x_i\kk\|_{P_i}^2\Big{\}}
     \label{Eq_Subproblem}
\end{eqnarray}
\EndFor
\State Update the dual variables:
\begin{equation}
     y^{(k+1)} \leftarrow y^{(k)} - \gamma\rho(Ax^{(k+1)}-b)
\end{equation}
\State Set $k \gets k+1$.
\EndWhile
\end{subequations}
  \end{algorithmic}
\end{algorithm}

In~\cite{Deng14}, the authors establish global convergence of J-ADMM for appropriately chosen regularization matrices $P_i$.  Moreover, they showed that J-ADMM has a convergence rate of o(1/k).


\subsection{Our Main Contributions}
We now summarize the main contributions of this work.
\begin{enumerate}
  \item We present a new flexible ADMM algorithm, called F-ADMM,
    that solves problems of the form \eqref{Problem} for strongly
    convex $f_i$, for general $n \geq 2$ based on a
    \emph{Gauss-Seidel} updating scheme. The quadratic regularizer
    used in F-ADMM is a user defined matrix that must be
    sufficiently positive definite (see Assumption \ref{Assume_Pi}),
    which makes F-ADMM flexible. For some applications, a
    careful choice of the regularizer makes the subproblems arising
    within F-ADMM significantly easier to solve, e.g., see the
    discussion in \cite[Section~1.2]{Deng12} and
    \cite[Section~1.2]{Deng14}. We prove that F-ADMM is \emph{globally
      convergent} in Section \ref{S_GADMM}. 
\item
We introduce a hybrid Jacobi/Gauss-Seidel variant of F-ADMM, called
H-ADMM, that is \emph{partially parallelizable}. This is significant
because it makes H-ADMM competitive in a \emph{big data setting}.
For H-ADMM, the blocks of variables are gathered into multiple groups,
with a Gauss-Seidel updating scheme between groups, and a Jacobi
updating scheme on the individual blocks within each group. We
demonstrate that H-ADMM is simply F-ADMM with a particular choice of
regularization matrix, and thus the convergence of H-ADMM follows directly from the convergence proof for F-ADMM.
\item
We show that if the $n$ blocks of data are partitioned into two
groups, then H-ADMM can be applied to convex functions $f_i$, rather
than strongly convex functions. In this special case, with carefully chosen regularization matrices, H-ADMM extends the algorithm in \cite{Deng12} from the $n=2$ case, to the case with general $n$, and convergence follows directly from the results presented in \cite{Deng12}.
\end{enumerate}

\subsection{Paper Outline}
In Section \ref{S_GADMM} we present our new flexible ADMM framework and show that any instance of it is globally convergent.  In Section \ref{S_HADMM} we consider a particular instance of our general framework, and proceed to show that it is a hybrid of Jacobi- and Gauss-Seidel-type updates.  We also discuss the practical advantages of this hybrid algorithm, which includes the fact that it is \emph{partially parallelizable}. Finally, in Section \ref{S_numericalresults} we present numerical experiments that illustrate the advantages of our flexible ADMM framework.

\section{A Flexible ADMM  (F-ADMM)}\label{S_GADMM}

In this section we present and analyze a new F-ADMM framework for solving problems of the form \eqref{Problem}.
%
%
For convenience, we define the vector
\begin{equation}\label{u}
  u\kk \eqdef \begin{bmatrix}
    x\kk\\
    y\kk
  \end{bmatrix}.
\end{equation}
Our analysis requires several assumptions concerning problem~\eqref{Problem} that are assumed to hold throughout.  The first of which uses $\partial f(x)$ to denote the subdifferential of $f$ at the point $x$, i.e.,
\begin{equation}
  \partial f(x) \eqdef \{s \in \R^N \;|\; \langle s, w-x\rangle \leq f(w) - f(x), \;\; \forall w \in {\rm dom} f\},
\end{equation}
where ${\rm dom} f = \{ x: f(x) < \infty\}$. Moreover,
\begin{equation}
  \partial f_i(x_i) \eqdef \{s_i \in \R^{N_i} \;|\; \langle s_i, w_i-x_i\rangle \leq f_i(w_i) - f_i(x_i), \;\; \forall w_i \in {\rm dom} f_i\}.
\end{equation}
We also require the following definition of strong convexity. A function $f_i: \R^N \to \R \cup \{+ \infty\}$ is strongly convex with convexity parameter $\mu_i  > 0$ if for all $x_i,w_i \in \dom f_i$,
\begin{equation}
\label{strongly_convex_1}
     f_i(w_i) \geq f_i(x_i) + \langle \partial f_i(x_i),w_i-x_i \rangle + \frac{\mu_i}{2}\|w_i-x_i\|_2^2.
\end{equation}
We may now state our assumptions on problem~\eqref{Problem}.

\begin{assumption}\label{Assume_saddlepoint}
  The set of saddle points (equivalently, the set of KKT-points) for~\eqref{Problem} is nonempty, i.e.,
  $$
  U^* := \{u^* \in\R^{N+m}: u^* = (x^*,y^*), \, A^Ty^* \in \partial f(x^*),
  \ \text{and} \
  Ax^* - b = 0\}
  \neq \emptyset.
  $$
\end{assumption}
\begin{assumption}\label{Assume_convex}
  The function $f_i$ is strongly convex with strong convexity constant $\mu_i>0$ for $i = 1,\dots,n$.
\end{assumption}

If Assumption \ref{Assume_saddlepoint} does not hold, then ADMM may
have unsolvable or unbounded subproblems, or the sequence of Lagrange
multiplier estimates may diverge. In particular, $x^*$ is the
solution to~\eqref{Problem} and $y^*$ is a solution to the associated dual
problem.  Assumption~\ref{Assume_convex} allows us to define
\begin{equation} \label{def-mu}
  \mu := \min_{1\leq i \leq n} \mu_i > 0
\end{equation}
as the minimum strong convexity parameter for the functions
$\{f_i\}_{i=1}^n$, as well as use the following lemma.
%
%
%
\begin{lemma}[Strong monotonicity of the subdifferential, Theorem~12.53 and Exercise~12.59 in \cite{Rockafellar09}]\label{Monosubgrad}
  Under Assumption \ref{Assume_convex}, for any $x_i,w_i \in {\rm dom} \,f_i$ we have
  \begin{equation}
    \langle s_i - t_i , x_i - w_i\rangle \geq \mu_i \|x_i - w_i\|_2^2,
    \qquad \forall s_i \in \partial f_i(x_i),\; t_i \in \partial f_i(w_i), \quad i = 1,\dots,n.
  \end{equation}
\end{lemma}

The following matrices will be important for defining the
regularization matrices used in our algorithm, and will also be used
in our convergence proof. In particular, we define the block diagonal matrix $A_D$, and the strictly upper triangular matrix $A_\bigtriangleup$ as
  \begin{equation}\label{ADandAT}
    A_\bigtriangleup \eqdef
    \begin{bmatrix}
       & & A_2 & \dots &A_n\\
       &  & &\ddots &\vdots\\
       & & & &A_n\\
       & &  & &
    \end{bmatrix} \quad \text{and} \quad
    A_D \eqdef
    \begin{bmatrix}
       A_1& &   \\
       & \ddots & \\
       & &  A_n
    \end{bmatrix},
  \end{equation}
  where $\{A_\bigtriangleup,A_D\} \subset \R^{mn \times N}$. We then have the strictly (block) upper triangular matrix
\begin{equation}\label{ADAT}
  A_D^TA_\bigtriangleup =  \begin{bmatrix}
       & & A_1^TA_2 & \dots &A_1^TA_n\\
       &  & &\ddots &\vdots\\
       & & & &A_{n-1}^TA_n\\
       & &  & &
    \end{bmatrix} \in \R^{N \times N}.
\end{equation}
Notice that $A_D^TA_\bigtriangleup$ is equivalent to ${ \rm
  triu}^+\!(A^TA),$ where ${ \rm triu}^+\!(X)$ denotes the strictly
upper (block) triangular part of $X$. We are now in a position to describe the details of our F-ADMM method.

\subsection{The Algorithm}\label{SS_GADMM}

Our F-ADMM method is stated formally as Algorithm~\ref{GSADMM}. As for
J-ADMM, F-ADMM requires the choice of a penalty parameter $\rho > 0$ and
regularization matrices $\{P_i\}_{i=1}^n$.  Our convergence analysis
considered in Section~\ref{SS_Convergence} requires them to satisfy the following
assumption that uses the definition of $\mu$ in~\eqref{def-mu}.

\begin{assumption}\label{Assume_Pi}
  The matrices $P_i$ are symmetric and satisfy $P_i \succ \frac{\rho^2}{2\mu}\|A_D^TA_\bigtriangleup\|_2^2 I$ for all $i = 1,\dots,n$.
\end{assumption}


\begin{algorithm}[H]
\caption{F-ADMM for solving problem~\eqref{Problem}.}\label{GSADMM}
  \begin{algorithmic}[1]
    \State \textbf{Initialize:} $x^{(0)}\in\R^N$, $y^{(0)} \in \R^m$,
    iteration counter $k=0$, parameters $\rho>0$, $\gamma \in(0,2)$,
    and matrices $\{P_i\}_{i=1}^n$ satisfying Assumption~\ref{Assume_Pi}.
    \While {stopping condition has not been met}
    \State Update the primal variables in a Gauss-Seidel fashion:\label{step-subproblem}
    \begin{eqnarray*}
      \displaystyle x_1^{(k+1)} \!\!&\leftarrow& \!\!\arg \min_{x_1} \Big{\{}f_1(x_1) +  \frac{\rho}2\|A_1x_1 + \sum_{j =2}^n A_jx_j^{(k)} - b - \frac{y\kk}{\rho} \|_2^2  + \frac{1}{2}\|x_1-x_1\kk\|_{P_1}^2\Big{\}}\\
      &\vdots&\\
      \displaystyle x_i^{(k+1)} \!\!&\leftarrow& \!\!\arg  \min_{x_i} \Big{\{}f_i(x_i) +  \frac{\rho}2\|A_ix_i + \sum_{j=1}^{i-1} A_jx_j\kpo + \sum_{l=i+1}^{n} A_lx_l\kk - b - \frac{y\kk}{\rho} \|_2^2 + \frac{1}{2}\|x_i-x_i\kk\|_{P_i}^2\Big{\}}\\
      &\vdots&\\
      \displaystyle x_n^{(k+1)} \!\!&\leftarrow& \!\!\arg  \min_{x_n} \Big{\{}f_n(x_n) +  \frac{\rho}2\|A_nx_n + \sum_{j =1}^{n-1} A_jx_j^{(k+1)} - b - \frac{y\kk}{\rho} \|_2^2  + \frac{1}{2}\|x_n-x_n\kk\|_{P_n}^2\Big{\}}
\end{eqnarray*}
\State Update the dual variables: \label{step-yupdate}
\begin{equation}\label{GSADMM_multupdate}
     y^{(k+1)} \leftarrow y^{(k)} - \gamma\rho(Ax^{(k+1)}-b )
\end{equation}
\State Set $k \gets k+1$.
\EndWhile
  \end{algorithmic}
\end{algorithm}

We now describe the $k$th iteration of Algorithm \ref{GSADMM} in more
detail. For fixed dual vector $y\kk$, the current point $x\kk$ is
updated in a Gauss-Seidel (i.e., a cyclic block-wise) fashion.  To
begin, decision vectors $x_2\kk,\dots,x_n\kk$ are fixed, and the first subproblem in
Step~\ref{step-subproblem} is minimized with respect to $x_1$ to give the new point $x\kpo$. Similar to before, we note that the $i$th subproblem in Step~\ref{step-subproblem} is equivalent to
\begin{equation}\label{FADMM_Lagrangeupdate}
  \displaystyle x_i^{(k+1)} \leftarrow \arg \min_{x_i} \, \mathcal{L}_\rho(x_1\kpo,\dots,x_{i-1}\kpo,x_i,x_{i+1}\kk,\dots,x_n\kk;y^{(k)})  + \tfrac12\|x_i - x_i^{(k)}\|_{P_i}^2.
\end{equation}
Next, the second block $x_2$ is updated using the information obtained
in the update of the first block $x_1$. That is, the vectors
$x_3\kk,\dots,x_n\kk$ remain fixed, as does $x_1\kpo$, and the
regularized augmented Lagrangian is minimized with respect to $x_2$ to
give the new point $x_2\kpo$. The process is repeated until all $n$
blocks have been updated, giving the vector $x\kpo$. Finally, the dual vector $y\kk$ is updated using the same formula  as in J-ADMM (see Algorithm~\ref{JADMM}).
Steps~\ref{step-subproblem} and \ref{step-yupdate} are repeated until a stopping threshold has been reached.

\begin{remark}
  It is clear that Algorithm \ref{GSADMM} uses a (serial) cyclic block
  coordinate descent (CD) type method to update the primal vector
  $x$. That is, in Step~\ref{step-subproblem} of Algorithm
  \ref{GSADMM}, a single pass of block CD is applied to the current point $x\kk$ to give the new point $x\kpo$, and then the dual vector is updated.
\end{remark}

\subsection{Convergence}\label{SS_Convergence}

To analyze F-ADMM, we require the block diagonal matrices $G_x$ and
$G$ defined as
\begin{equation}\label{G}
  G_x \eqdef \begin{bmatrix}
    P_1 & &\\
    &\ddots & \\
    & & P_n
  \end{bmatrix} \quad \text{and} \quad
  G \eqdef \begin{bmatrix}
    G_x & \\
    & \frac{1}{\gamma \rho}I
  \end{bmatrix},
\end{equation}
where $I$ is the (appropriately sized) identity matrix, and $\gamma
\in (0,2)$ and $\rho >0$ are algorithm parameters. The following
result gives sufficient conditions for declaring that a limit point of
problem~\eqref{Problem} is optimal. 

\begin{lemma} \label{lem:fp}
 If $\mathcal{K}$ is any subsequence of the natural numbers
 satisfying
\begin{equation}\label{Eq_limit}
  \lim_{k\in\mathcal{K}} u\kk
  = u^{L}
  \ \ \text{and} \ \
  \lim_{k\in\mathcal{K}} \|u\kk - u\kpo\|_G = 0
\end{equation}
 for some limit point $u^{L}$, then $u^{L} \in U^*$, i.e., $u^{L}$
 solves problem~\eqref{Problem}.
\end{lemma}
\begin{proof}
Let us first observe that the two limits in~\eqref{Eq_limit} jointly
imply that
\begin{equation} \label{Eq_limit_kpo}
  \lim_{k\in\mathcal{K}} u\kpo
  = u^L
  \equiv \begin{pmatrix} x^L \\ y^L \end{pmatrix}.
\end{equation}
Also, it follows from \eqref{Eq_limit} and the definitions of $u\kk$
(see~\eqref{u}) and $G$ (see~\eqref{G}), that $\lim_{k\in\mathcal{K}} (y\kk
- y\kpo) = 0$. Combining this
with \eqref{GSADMM_multupdate},
\eqref{Eq_limit_kpo}, and
\eqref{Eq_limit}
establishes that
\begin{equation}\label{constraintsatisfied}
  b
  = \lim_{k\in\mathcal{K}} A x\kpo
  = Ax^L
  = \lim_{k\in\mathcal{K}} A x\kk
\end{equation}
so that, in particular, $x^L$ is feasible for problem~\eqref{Problem}.


Next, the optimality condition for the $i$th subproblem in
Step~\ref{step-subproblem} of Algorithm~\ref{GSADMM}
ensures the existence of a vector
$g_i(x_i^{(k+1)}) \in \partial f_i(x_i\kpo)$ satisfying
%
\begin{align*}
  0 &= g_i(x_i^{(k+1)}) + \rho A_i^T\Big(A_ix_i\kpo + \sum_{j=1}^{i-1 }A_jx_j\kpo + \sum_{l=i+1}^{n}A_lx_l^{(k)} - b - \frac{y\kk}{\rho}\Big) + P_i(x_i\kpo - x_i\kk )\\
    &= g_i(x_i^{(k+1)}) - A_i^T y\kk + \rho A_i^T\Big(A_ix_i\kpo + \sum_{j=1}^{i-1 }A_jx_j\kpo + \sum_{l=i+1}^{n}A_lx_l^{(k)} - b \Big) + P_i(x_i\kpo - x_i\kk )\\
    &= g_i(x_i^{(k+1)}) - A_i^T y\kk + \rho A_i^T\Big(\sum_{j=1}^{i}A_jx_j\kpo + \sum_{l=i+1}^{n}A_lx_l^{(k)} - Ax^L \Big) + P_i(x_i\kpo - x_i\kk ),
\end{align*}
where we also used~\eqref{constraintsatisfied} to substitute for $b$ in the last equation. Using  $Ax^L = \sum_{j=1}^n A_jx_j^L$ and rearranging the previous equation gives
$$
g_i(x_i^{(k+1)})
=  A_i^T y\kk - \rho A_i^T\Big(\sum_{j=1}^{i}A_j(x_j\kpo-x_j^L) + \sum_{l=i+1}^{n}A_l(x_l\kk-x_l^L) \Big) - P_i(x_i\kpo - x_i\kk ).
$$
By taking limits over the subsequence $\mathcal{K}$ of the previous equation, and using \eqref{Eq_limit} and~\eqref{Eq_limit_kpo}, we know that
\begin{equation} \label{sg-limit}
  \lim_{k\in\mathcal{K}} g_i(x_i^{(k+1)}) = A_i^T y^L.
\end{equation}
We may then
use $g_i(x_i^{(k+1)}) \in \partial f_i(x_i\kpo)$, \eqref{Eq_limit_kpo}, \eqref{sg-limit}, and~\cite[Theorem~24.4]{Rockafellar70} to conclude that
$$
A_i^T y^L \in \partial f_i(x_i^{L}).
$$
Combining this inclusion, which holds for all $1 \leq i \leq n$,  with
\eqref{constraintsatisfied} shows that $u^{L}$ is a KKT point for problem \eqref{Problem}, and thus is a solution as claimed.
\end{proof}

Our aim is to combine Lemma~\ref{lem:fp} with the next
result, which shows that the sequence $\{\|u_k-u^*\|_G\}$ is nonexpansive with
respect to any $u^* \in U^*$. We
note that the proof is inspired by that for
J-ADMM~\cite{Deng14}.
\begin{theorem}\label{Thm_convergence}
  Let Assumptions \ref{Assume_saddlepoint}, \ref{Assume_convex},
  and \ref{Assume_Pi} hold.  Then, for any $u^*\in
  U^*$ and all $k \geq 1$,  there exists a constant $\eta >0$ such that
  \begin{equation}\label{contraction}
    \|u\kk - u^*\|_G^2 - \|u\kpo - u^*\|_G^2 \geq \eta \|u\kk - u\kpo\|_G^2
  \end{equation}
  with $u\kk$ defined in \eqref{u} and $G$ defined in \eqref{G}.
\end{theorem}
\begin{proof}
At each iteration of Algorithm \ref{GSADMM}, a subproblem of the following form is solved for $x_i$:
\begin{eqnarray}\label{Alg_xi_subproblem}
  x_i^{(k+1)} = \arg \min_{x_i} \Big{\{}f_i(x_i) + \frac{\rho}{2}\|A_ix_i + \sum_{j=1}^{i-1}A_jx_j^{(k+1)} + \sum_{l=i+1}^{n}A_lx_l^{(k)} - b - \frac{y\kk}{\rho}\|_2^2 + \frac12\|x_i - x_i\kk\|_{P_i}\Big{\}}.
\end{eqnarray}
The first order optimality condition for \eqref{Alg_xi_subproblem} is
\begin{eqnarray*}
  0 \in \partial f_i(x_i^{(k+1)}) + \rho A_i^T\left(\sum_{j=1}^{i}A_jx_j\kpo + \sum_{l=i+1}^{n}A_lx_l^{(k)} - b - \frac{y\kk}{\rho}\right) + P_i(x_i\kpo - x_i\kk ),
\end{eqnarray*}
and rearranging gives
\begin{eqnarray*}
   \rho A_i^T\left(b + \frac{y\kk}{\rho} - \sum_{j=1}^{i}A_jx_j\kpo - \sum_{l=i+1}^{n}A_lx_l^{(k)}\right) + P_i(x_i\kk - x_i\kpo ) \in \partial f_i(x_i^{(k+1)}).
\end{eqnarray*}
Noting that $\sum_{l=i+1}^nA_l x_l\kpo - Ax\kpo = - \sum_{j=1}^iA_jx_j\kpo$ and defining $\hat{y} \eqdef y\kk - \rho(Ax^{(k+1)}-b)$ gives
\begin{eqnarray*}
   A_i^T\hat{y} - \rho A_i^T\left(\sum_{j=i+1}^{n}A_j(x_j\kk - x_j\kpo)\right) + P_i(x_i\kk - x_i\kpo ) \in \partial f_i(x_i^{(k+1)}).
\end{eqnarray*}
Using Lemma \ref{Monosubgrad} with $A_i^Ty^* \in \partial f_i(x_i^*)$, we have
\begin{align*}
  \mu_i\|x_i\kpo - x_i^*\|_2^2
  &\leq
  \left\langle x_i^{(k+1)}-x_i^*,A_i^T(\hat{y}- y^*) - \rho A_i^T\!\!\sum_{j=i+1}^{n}\!A_j(x_j\kk - x_j\kpo) +P_i(x_i\kk - x_i\kpo )\right\rangle.
\end{align*}
Now, summing the previous inequality over all blocks $i$ gives
\begin{align}
  \mu\|x\kpo - x^*\|_2^2
  &\leq \left\langle A(x^{(k+1)}-x^*),\hat{y} - y^* \right\rangle -
  \rho\sum_{i=1}^n\Big{\langle}
  A_i(x_i^{(k+1)}-x_i^*),\sum_{j=i+1}^{n}A_j(x_j\kk -
  x_j\kpo)\Big{\rangle} \nonumber \\
  &+  \sum_{i=1}^n(x_i^{(k+1)}-x_i^*)^TP_i(x_i\kk - x_i\kpo ),
  \label{Eq_innerprod2}
\end{align}
where $\mu > 0$ is defined in~\eqref{def-mu}.
Notice that, by \eqref{GSADMM_multupdate} the following relation holds
\begin{equation}\label{Alambdarelation}
  A(x^{(k+1)}-x^*) = \tfrac{1}{\gamma \rho}(y\kk-y\kpo),
\end{equation}
and we also have that
\begin{equation}\label{Lambdarelation}
  \hat{y} - y^* = (\hat{y} - y\kpo) + (y\kpo - y^*) = \frac{\gamma-1}{\gamma}(y\kk - y\kpo)+(y\kpo - y^*).
\end{equation}
Then \eqref{Eq_innerprod2} becomes
\begin{eqnarray}\label{Eq_PAinnerprod}
\notag
  && \sum_{i=1}^n(x_i^{(k+1)}-x_i^*)^TP_i(x_i\kk - x_i\kpo ) - \rho\sum_{i=1}^n\Big{\langle} A_i(x_i^{(k+1)}-x_i^*),\sum_{j=i+1}^{n}A_j(x_j\kk - x_j\kpo)\Big{\rangle}\\
   \notag
   &\overset{\eqref{Alambdarelation}}{\geq}& \mu\|x\kpo - x^*\|_2^2-\frac{1}{\gamma \rho}\langle y^{(k)}-y^{(k+1)},\hat{y} - y^* \rangle \\
  &\overset{\eqref{Lambdarelation}}{=}& \mu\|x\kpo - x^*\|_2^2-\frac{1}{\gamma \rho}\langle y^{(k)}-y^{(k+1)},y\kpo - y^* \rangle  +  \frac{1-\gamma}{\gamma^2 \rho}\| y\kk - y\kpo \|_2^2.
\end{eqnarray}
Using the identity
  \begin{equation}\label{sum_to_matrix}
    \sum_{i=1}^n\Big{\langle} A_i(x_i^{(k+1)}-x_i^*), \sum_{j=i+1}^{n}A_j(x_j\kk-x_j\kpo)\Big{\rangle} \overset{\eqref{ADAT}}{=}
       \Big{\langle} x^{(k+1)}-x^*,A_D^TA_\bigtriangleup(x\kk-x\kpo) \Big{\rangle},
  \end{equation}
we may deduce from~\eqref{Eq_PAinnerprod} and the definition of $G_x$ that
\begin{eqnarray}\label{sum_to_matrix2}
\notag
  && \sum_{i=1}^n(x_i^{(k+1)}-x_i^*)^TP_i( x_i\kk-x_i\kpo) \\
  \notag
  &=& \big{\langle} x^{(k+1)}-x^*,G_x(x\kk-x\kpo) \big{\rangle}\\
  \notag
  &\geq& \mu\|x\kpo - x^*\|_2^2-\frac{1}{\gamma \rho}\langle y^{(k)}-y^{(k+1)},y\kpo - y^* \rangle \\
  && + \,\frac{1-\gamma}{\gamma^2 \rho}\| y\kk - y\kpo \|_2^2 + \rho\Big{\langle} x^{(k+1)}-x^*,A_D^TA_\bigtriangleup(x\kk-x\kpo) \Big{\rangle}.
\end{eqnarray}
Then, by rearranging
\eqref{sum_to_matrix2} we have
\begin{align}
 (u\kpo - u^*)G(u\kk - u\kpo)
 &\geq \mu\|x\kpo - x^*\|_2^2 +  \frac{1-\gamma}{\gamma^2 \rho}\| y\kk
  - y\kpo \|_2^2 \nonumber\\
 &+ \rho\Big{\langle} x^{(k+1)}-x^*,A_D^TA_\bigtriangleup(x\kk-x\kpo) \Big{\rangle},
  \label{Glowerbound}
\end{align}
where $G$ is defined in \eqref{G}. Combining the relation
\begin{equation*}
   \|u\kk - u^*\|_G^2 - \|u\kpo - u^*\|_G^2 = 2(u\kpo - u^*)G(u\kk - u\kpo) + \|u\kk - u\kpo\|_G^2
\end{equation*}
with \eqref{Glowerbound} gives
\begin{align}
  \|u\kk - u^*\|_G^2 - \|u\kpo - u^*\|_G^2 
  &\geq 2\mu\|x\kpo - x^*\|_2^2 +  2\frac{1-\gamma}{\gamma^2 \rho}\|
  y\kk - y\kpo \|_2^2 + \|u\kk - u\kpo\|_G^2 \nonumber\\
  &+ 2\rho\Big{\langle}
  x^{(k+1)}-x^*,A_D^TA_\bigtriangleup(x\kk-x\kpo) \Big{\rangle} \nonumber\\
  &= 2\mu\|x\kpo - x^*\|_2^2 + \frac{2-\gamma}{\gamma^2 \rho}\| y\kk -
  y\kpo \|_2^2 + \|x\kk - x\kpo\|_{G_x}^2 \nonumber\\
  &+ 2\rho\Big{\langle}
  x^{(k+1)}-x^*,A_D^TA_\bigtriangleup(x\kk-x\kpo) \Big{\rangle}.
  \label{exp1}
\end{align}
Notice that, because $\mu>0$, the following holds:
  \begin{equation}\label{exp2}
    2\rho\Big{\langle} x^{(k+1)}-x^*,A_D^TA_\bigtriangleup(x\kk-x\kpo) \Big{\rangle}
    \geq
    -2\mu\|x^{(k+1)}-x^*\|_2^2 - \frac{\rho^2}{2\mu}\|A_D^TA_\bigtriangleup(x\kk-x\kpo)\|_2^2.
  \end{equation}
Now, combining \eqref{exp1} and \eqref{exp2} gives
  \begin{eqnarray}
   \|u\kk - u^*\|_G^2 &-& \|u\kpo - u^*\|_G^2 \nonumber\\
   &\geq& \frac{2-\gamma}{\gamma^2 \rho}\| y\kk
   - y\kpo \|_2^2 + \|x\kk - x\kpo\|_{G_x}^2
   - \frac{\rho^2}{2\mu}\|A_D^TA_\bigtriangleup\|_2^2\|x\kk-x\kpo\|_2^2 \nonumber\\
   &=& \frac{2-\gamma}{\gamma^2 \rho}\| y\kk - y\kpo \|_2^2  + \sum_{i=1}^n\|x_i\kk - x_i\kpo\|_{P_i - \frac{\rho^2}{2\mu}\|A_D^TA_\bigtriangleup\|_2^2I}^2
   \label{hi}
  \end{eqnarray}
  and note that Assumption \ref{Assume_Pi} guarantees that
  $$
  T_i
  := P_i - \frac{\rho^2}{2\mu}\|A_D^TA_\bigtriangleup\|_2^2I \succ 0.
  $$
  If we then let $\eta_i := \lambda_{min}(T_i)/\|P_i\|_2 > 0$, we
  have from the definition of $T_i$ and standard norm inequalities
  \begin{align*}
  \|x_i\kk - x_i\kpo\|_{P_i - \frac{\rho^2}{2\mu}\|A_D^TA_\bigtriangleup\|_2^2I}^2
  &= \|x_i\kk - x_i\kpo\|_{T_i}^2 \\
  &\geq \lambda_{min}(T_i)\|x_i\kk - x_i\kpo\|_2^2 \\
  &= \eta_i\|P_i\|_2 \|x_i\kk - x_i\kpo\|_2^2 \\
  &\geq \eta_i (x_i\kk - x_i\kpo)^T P_i(x_i\kk - x_i\kpo)
  = \eta_i \|x_i\kk - x_i\kpo\|_{P_i}^2.
  \end{align*}
  Combining this with~\eqref{hi} gives
  $$
  \|u\kk - u^*\|_G^2 - \|u\kpo - u^*\|_G^2
  \geq \frac{2-\gamma}{\gamma^2 \rho}\| y\kk - y\kpo \|_2^2
  + \sum_{i=1}^n \eta_i \|x_i\kk - x_i\kpo\|_{P_i}^2.
  $$
  From the previous inequality and the definition
  $$
  \eta
  := \min\left\{\frac{2-\gamma}{\gamma}, \min_{1\leq i\leq n}\eta_i\right\} > 0,
  $$
  we have
  $$
  \|u\kk - u^*\|_G^2 - \|u\kpo - u^*\|_G^2
  \geq \eta \left(\frac{1}{\gamma \rho}\| y\kk - y\kpo \|_2^2
  + \sum_{i=1}^n \|x_i\kk - x_i\kpo\|_{P_i}^2\right)
  = \eta\|u\kk - u\kpo\|_G^2,
  $$
  which is the desired result.
\end{proof}

We we may now state our main convergence result for Algorithm~\ref{GSADMM}.
\begin{theorem}
 If the conditions of Theorem \ref{Thm_convergence} hold, then the
 sequence $\{u\kk\}_{k\geq 0}$ generated by Algorithm \ref{GSADMM}
 converges to some vector $u^L$ that is a solution to problem \eqref{Problem}.
\end{theorem}
\begin{proof}
  Let $u^*$ be \emph{any} solution in $U^*$.  It then follows from
  Theorem~\ref{Thm_convergence} that
  \begin{equation} \label{eq:bounded}
    \|u\kk - u^*\|_G \leq \|u^{(0)} - u^*\|_G \ \text{for all $k \geq 1$,}
  \end{equation}
  so that $\{u\kk\}_{k\geq 0}$ is a bounded sequence.  Moreover, for any integer
  $p \geq 1$, it follows from~\eqref{Thm_convergence} that
  $$
  \sum_{k=1}^p \eta\|u\kk - u\kpo\|_G^2
  \leq  \sum_{k=1}^p \left( \|u\kk - u^*\|_G^2 - \|u\kpo-u^*\|_G\right)
  = \|u^{(0)} - u^*\|_G^2 - \|u^{(p+1)} - u^*\|_G^2
  \leq \|u^{(0)} - u^*\|^2_G.
  $$
  Taking limits of both sides of the previous inequality as
  $p\to\infty$ shows that the sum is finite, and since all the
  summands are nonnegative that
  \begin{equation} \label{eq:lim-diff}
    \lim_{k\to\infty} \|u\kk - u\kpo\|_G = 0.
  \end{equation}

  Next, using the boundedness of $\{u\kk\}_{k\geq 0}$, we may conclude the existence
  of a subsequence $\mathcal{K}\subseteq \{1,2,\dots\}$ and a vector
  $u^L\in\R^{N+m}$ such that
  \begin{equation} \label{eq:u-limit}
     \lim_{k\in\mathcal{K}} u\kk = u^L.
  \end{equation}
  It follows from~\eqref{eq:u-limit}, \eqref{eq:lim-diff}, and Lemma~\ref{lem:fp}
  that $u_L$ is a solution to
  problem~\eqref{Problem}. Finally, since~\eqref{contraction} held for
  \emph{any} $u^*\in U^*$ and we have proved that $u^L\in U^*$, it
  follows that $\lim_{k\to\infty} u\kk = u^L$, as desired.
\end{proof}

\section{A Hybrid ADMM (H-ADMM)}\label{S_HADMM}
One of the disadvantages of a Gauss-Seidel type updating scheme within ADMM is that it is inherently serial. With problem dimension growing ever larger in this era of big data, and the ubiquity of parallel processing power, a Jacobi type updating scheme may be preferable in many real-world instances of problem \eqref{Problem}. The purpose of this section is to show that if F-ADMM is applied to ``grouped data", and a special choice of regularization matrix is employed for each group, then Algorithm \ref{GSADMM} becomes a \emph{hybrid} Gauss-Seidel/Jacobi ADMM-type method. Therefore, Algorithm \ref{GSADMM} is \emph{partially parallelizable}.

\subsection{Notation and Assumptions}

Suppose that the function $f(x)$ is separable into $n$ blocks, as in
\eqref{Problem_objective}. Then, we can (implicitly) partition the
variables $x_i$ and functions $f_i(x_i)$ together into $\ell < n$
groups. For simplicity of exposition, we will assume that $n$ is
divisible by some $p$, so that $\ell p = n$, which means that we form
$\ell$ groups of $p$ blocks. Then, problem \eqref{Problem} is
equivalent to the following partitioned problem: 
\begin{subequations}\label{Problem_lblockform}
  \begin{eqnarray}
  &\displaystyle \minimize{x\in \R^{N}}& f(x) \equiv \sum_{j=1}^{\ell} \mathbf{f}_j(\mathbf{x}_j)\\
  &\subject& \sum_{j=1}^{\ell} \mathcal{A}_j\mathbf{x}_j = b
\end{eqnarray}
\end{subequations}
with
\begin{eqnarray}\label{Groupedx}
  \mathbf{x}_1 \eqdef \begin{bmatrix}
    x_1\\ \vdots\\ x_{p}
  \end{bmatrix},\quad
  \mathbf{x}_2 \eqdef \begin{bmatrix}
    x_{p+1}\\ \vdots\\ x_{2p}
  \end{bmatrix},\quad
  \dots \quad
  \mathbf{x}_{\ell} \eqdef \begin{bmatrix}
    x_{(\ell-1)p + 1}\\ \vdots\\ x_{n}
  \end{bmatrix},
\end{eqnarray}
\begin{eqnarray*}
  \mathbf{f}_1(\mathbf{x}_1) \eqdef \sum_{i=1}^{p} f_i(x_i), \quad \mathbf{f}_2(\mathbf{x}_2) \eqdef \sum_{i=p+1}^{2p} \!f_i(x_i), \quad\dots \quad \mathbf{f}_{\ell}(\mathbf{x}_{\ell}) \eqdef \sum_{i=(\ell-1)p+1}^{n} \!\!\!\!\!f_i(x_i),
\end{eqnarray*}
and
\begin{eqnarray*}
  \mathcal{A}_1 \eqdef \begin{bmatrix}
    A_1 & \dots & A_{p}
  \end{bmatrix}, \quad
   \mathcal{A}_2 \eqdef \begin{bmatrix}
    A_{p+1} & \dots & A_{2p}
  \end{bmatrix},\quad \dots \quad
  \mathcal{A}_{\ell} \eqdef \begin{bmatrix}
    A_{(n-1)p+1} & \dots & A_{n}
  \end{bmatrix}.
\end{eqnarray*}
Notice that $A = [\mathcal{A}_1, \mathcal{A}_2,\dots \mathcal{A}_{\ell}] \equiv [A_1,A_2,\dots,A_n]$, and $x = [\mathbf{x}_{1}^T,\dots,\mathbf{x}_{\ell}^T]^T \equiv [x_1^T,\dots,x_n^T]^T$. Furthermore, it will be useful to define the index sets
\begin{equation}\label{Groupedindexsets}
\mathcal{S}_1 = \{1,\dots,p\}, \quad \mathcal{S}_2 = \{p+1,\dots,2p\}, \quad \dots\quad \mathcal{S}_{\ell} = \{(\ell-1)p+1,\dots,n\}
\end{equation}
associated with the partition described above, and to use the notation $\mathcal{S}_{i,j}$ to denote the $j$th element of $\mathcal{S}_i$.


We now think of applying Algorithm \ref{GSADMM} to the $\ell$ groups
of data. That is, in Step 3 of Algorithm \ref{GSADMM} we have $\ell$
minimization problems, one for each of the grouped data points
$\mathbf{x}_j$ (rather than $n$ minimization problems, one for each of
the individual data blocks $x_i$). For the grouped data, we require
regularization matrices $\p_1,\dots,\p_{\ell}$, for each of the $\ell$
groups; these matrices will be crucial in our upcoming derivation.

To motivate the idea of ``grouped data'', and to make the ideas that
will be discussed in this rest of this section more concrete, we give
a specific example that shows how our hybrid algorithm will work.
\begin{example}\label{Example_GSJ}
  Suppose there are $n = 12$ blocks and we have access to a parallel computer with $p = 4$ processors. We make a formal partition of the data into $\ell = 3$ groups. That is, we set $\mathbf{f}_1(\mathbf{x}_1) = \sum_{i=1}^4f_i(x_i)$, $\mathbf{f}_2(\mathbf{x}_2) = \sum_{i=5}^8f_i(x_i)$ and $\mathbf{f}_3(\mathbf{x}_3) = \sum_{i=9}^{12}f_i(x_i)$, $\mathbf{x}_1 = [x_1^T,\dots,x_4^T]^T$, $\mathbf{x}_2 = [x_5^T,\dots,x_8^T]^T$, and $\mathbf{x}_3 = [x_9^T,\dots,x_{12}^T]^T$, partition the matrix $A$ accordingly, and initialize index sets $\mathcal{S}_1 = \{1,\dots,4\}$, $\mathcal{S}_2 = \{5,\dots,8\}$ and $\mathcal{S}_3 = \{9,\dots,12\}$. Then, a single iteration of H-ADMM (see Steps 3--7 of Algorithm \ref{HADMM}) will run in the following way. The Lagrange multiplier estimate $y\kk$ and (group) variables $\mathbf{x}_2\kk$ and $\mathbf{x}_3\kk$ are fixed. Group variable $\mathbf{x}_1$ is updated by solving a subproblem of the form \eqref{HGADMM-parallelloop} \emph{for each of} $x_1,\dots,x_4$ \emph{in parallel}. This gives the new point $\mathbf{x}_1\kpo$. Then, $\mathbf{x}_1\kpo$ and $\mathbf{x}_3\kk$ are fixed, and four subproblems of the form \eqref{HGADMM-parallelloop} are solved for each of $x_5,\dots,x_8$ \emph{in parallel}, giving $\mathbf{x}_2\kpo$. Next, $\mathbf{x}_1\kpo$ and $\mathbf{x}_2\kpo$ are fixed, and four subproblems of the form \eqref{HGADMM-parallelloop} are solved for each of $x_9,\dots,x_{12}$ \emph{in parallel}, giving $\mathbf{x}_3\kpo$. Finally, $y^{(k+1)}$ is updated in~\eqref{eq:yup-bd}.
\end{example}
Example \ref{Example_GSJ} shows that Algorithm \ref{HADMM} is running
a Gauss-Seidel process on the group variables, but running a Jacobi
process to update the individual blocks within each group. This
example shows an efficient implementation in the sense that, by
ensuring that the group size $p$ matches the number of processors, all
processors are always engaged, and that updated information is utilized when it is available.

In the rest of this section we explain \emph{how} H-ADMM (Algorithm~\ref{HADMM}) is obtained from F-ADMM.

\subsection{Separability Via Regularization}

We show that, if the regularization matrices
$\{\p_i\}_{i=1}^\ell$ are chosen appropriately, F-ADMM can be
partially parallelized, and forms the hybrid algorithm H-ADMM. In
particular, for the $i$th subproblem in F-ADMM (applied
to the grouped data in \eqref{Problem_lblockform}), the $p$
blocks within the $i$th group can be solved for in parallel. 

In what follows, we use the relationships
\begin{align}
  \|\sum_{j=1}^nA_jx_j\|_2^2
  &= \sum_{j=1}^n \langle A_jx_j,A_jx_j \rangle + \sum_{j=1}^n\sum_{\overset{l=1}{l\neq j}}^n\langle A_jx_j,A_lx_l\rangle
  \label{norm_rel}
  \ \ \text{and} \\
  \sum_{j\in\Sscr_i} \sum_{\overset{l\in\Sscr_i}{l\neq j}} \langle A_l x_l, A_j x_j \rangle
  &= \sum_{j=1}^p \sum_{\overset{l=1}{l\neq j}}^p \langle A_{\Sscr_{i,j}} x_{\Sscr_{i,j}}, A_{\Sscr_{i,l}} x_{\Sscr_{i,l}}\rangle,
  \label{are-equal}
\end{align}
which can easily be verified.  Using the definition of $\Ascr_i$ and a similar reasoning as for~\eqref{norm_rel}, it follows that
\begin{eqnarray}\label{Eq_Aixi}
  \|\mathcal{A}_i\mathbf{x}_i\|_2^2 = \|\sum_{j\in \mathcal{S}_i} A_j x_j\|_2^2
  =\sum_{j\in \mathcal{S}_i} \| A_jx_j\|_2^2 + \sum_{j\in \mathcal{S}_i}\sum_{\overset{l \in \mathcal{S}_i}{l\neq j}} \langle A_jx_j,A_lx_l\rangle.
\end{eqnarray}
We now define
$
\mathbf{b}_i := b - \sum_{q=1}^{i-1}\mathcal{A}_q\mathbf{x}_q\kpo - \sum_{s=i+1}^{\ell}\mathcal{A}_s\mathbf{x}_s\kk,
$
and notice that $\mathbf{b}_i$ is fixed when minimizing the augmented Lagrangian with respect to group $\mathbf{x}_i$. Recalling Algorithm~\ref{GSADMM} and \eqref{FADMM_Lagrangeupdate}, and using~\eqref{Eq_Aixi}, the update for the $i$th subproblem for our grouped data problem \emph{without} the regularization term is equivalent to
\begin{align}
\!\mathbf{x}_i^{(k+1)} &= \arg \min_{\mathbf{x}_i} \, \mathcal{L}_\rho(\mathbf{x}_1\kpo,\dots,\mathbf{x}_{i-1}\kpo, \mathbf{x}_i, \mathbf{x}_{i+1}\kk,\dots,\mathbf{x}_{\ell}\kk;y^{(k)}) \nonumber\\
   &= \arg \min_{\mathbf{x}_i} \Big{\{} \mathbf{f}_i(\mathbf{x}_i) - \langle y^{(k)},\mathcal{A}_i\mathbf{x}_i -\mathbf{b}_i\rangle + \frac{\rho}{2}\|\mathcal{A}_i\mathbf{x}_i -\mathbf{b}_i\|_2^2\Big{\}} \nonumber\\
  &= \arg \min_{\mathbf{x}_i}\Big{\{} \mathbf{f}_i(\mathbf{x}_i) - \langle y^{(k)},\mathcal{A}_i\mathbf{x}_i -\mathbf{b}_i\rangle + \frac{\rho}2\|\mathbf{b}_i\|_2^2+ \frac{\rho}2\|\mathcal{A}_i\mathbf{x}_i\|_2^2  - \rho\langle \mathcal{A}_i\mathbf{x}_i,\mathbf{b}_i \rangle\Big{\}} \nonumber\\
  &= \arg \min_{\mathbf{x}_i}\Big{\{} \mathbf{f}_i(\mathbf{x}_i) - \langle y^{(k)}\!,\mathcal{A}_i\mathbf{x}_i -\mathbf{b}_i\rangle
  - \rho\langle \mathcal{A}_i\mathbf{x}_i,\mathbf{b}_i\rangle
  + \frac{\rho}2\sum_{j\in \mathcal{S}_i} \| A_jx_j\|_2^2   
+ \frac{\rho}2
 \sum_{j\in \mathcal{S}_i}
 \sum_{ \overset{l \in \mathcal{S}_i}{l\neq j}} \langle A_lx_l,A_jx_j\rangle \Big{\}}.\label{xitildeupdate}
\end{align}
Notice that it is the final term in \eqref{xitildeupdate} that makes the minimization of the augmented Lagrangian (with respect to the group $\mathbf{x}_i$) non-separable; it contains a cross product term, which shows interaction between different blocks of variables within the $i$th group indexed by $\mathcal{S}_i$.

\subsubsection{Defining the group regularization matrices}\label{SSS_Pi}
We eliminate the non-separability in \eqref{xitildeupdate} by carefully choosing the regularization matrices $\{\p_i\}_{i=1}^\ell$. From a practical perspective, if the problem is made separable, then the individual blocks within the $i$th group can be \emph{updated in parallel}. To this end, we choose the matrix that defines our regularizer to be
\begin{equation}\label{Pi}
  \p_i \eqdef \begin{bmatrix}
    P_{\mathcal{S}_{i,1}} &  -\rho A_{\mathcal{S}_{i,1}}^TA_{\mathcal{S}_{i,2}} &  &\dots &-\rho A_{\mathcal{S}_{i,1}}^TA_{\mathcal{S}_{i,p}}\\
   -\rho A_{\mathcal{S}_{i,2}}^TA_{\mathcal{S}_{i,1}} & P_{\mathcal{S}_{i,2}} & & &\vdots\\
   \vdots & &\ddots &  &\\
    \vdots & & & P_{\mathcal{S}_{i,p-1}}  &-\rho A_{\mathcal{S}_{i,p-1}}^TA_{\mathcal{S}_{i,p}} \\
   -\rho A_{\mathcal{S}_{i,p}}^TA_{\mathcal{S}_{i,1}} & &\dots& -\rho A_{\mathcal{S}_{i,p}}^TA_{\mathcal{S}_{i,p-1}} & P_{\mathcal{S}_{i,p}}
  \end{bmatrix}.
\end{equation}
We remind the reader that the matrices
$\{P_{\mathcal{S}_{i,j}}\}_{j=1}^p$ used to define $\p_i$ are user
defined symmetric matrices that must be chosen to be sufficiently
positive definite, to ensure that convergence of F-ADMM on the grouped
data is guaranteed. Before we formalize our assumption, we require the
definitions
\begin{equation}\label{ADandATscript}
    \mathcal{A}_\bigtriangleup \eqdef
    \begin{bmatrix}
       & & \mathcal{A}_2 & \dots &\mathcal{A}_{\ell}\\
       &  & &\ddots &\vdots\\
       & & & &\mathcal{A}_{\ell}\\
       & &  & &
    \end{bmatrix} \quad \text{and} \quad
    \mathcal{A}_D \eqdef
    \begin{bmatrix}
       \mathcal{A}_1& &   \\
       & \ddots & \\
       & &  \mathcal{A}_{\ell}
    \end{bmatrix},
  \end{equation}
  where $\{\mathcal{A}_\bigtriangleup,\mathcal{A}_D\} \subset \R^{m\ell \times N}$. We then have the strictly (block) upper triangular matrix
\begin{equation}\label{ADATscript}
  \mathcal{A}_D^T\mathcal{A}_\bigtriangleup =  \begin{bmatrix}
       & & \mathcal{A}_1^T\mathcal{A}_2 & \dots &\mathcal{A}_1^T\mathcal{A}_{\ell}\\
       &  & &\ddots &\vdots\\
       & & & &\mathcal{A}_{\ell-1}^T\mathcal{A}_{\ell}\\
       & &  & &
    \end{bmatrix} \in \R^{N \times N}.
\end{equation}
Notice that the definitions of $\mathcal{A}_D$,
$\mathcal{A}_\bigtriangleup$ and
$\mathcal{A}_D^T\mathcal{A}_\bigtriangleup$ in \eqref{ADandATscript}
and \eqref{ADATscript}, are analogues to $A_D$, $A_\bigtriangleup$ and
$A_D^TA_\bigtriangleup$defined in \eqref{ADandAT} and \eqref{ADAT}.
We are now ready to state our assumption on $\{\p_i\}_{i=1}^\ell$,
which is actually Assumption~\ref{Assume_Pi}
applied to the grouped data problem~\eqref{Problem_lblockform}.

\begin{assumption}\label{Assume_Pj}
  The matrices $\p_i$ are symmetric and satisfy $\p_i \succ \frac{\rho^2}{2\mu}\|\mathcal{A}_D^T\mathcal{A}_\bigtriangleup\|_2^2 I$ for all $i = 1,\dots,\ell$.
\end{assumption}

Importantly, if F-ADMM is applied to the grouped data problem
\eqref{Problem_lblockform} and Assumption~\ref{Assume_Pj} holds,
then convergence is automatic, i.e., convergence of F-ADMM equipped with Assumption~\ref{Assume_Pj} applied to problem \eqref{Problem_lblockform} follows directly from the convergence results presented in Section~\ref{S_GADMM}.

\subsubsection{Incorporating the regularization term}\label{SS_Pj}

Now that the regularization matrices $\{\p_i\}_{i=1}^\ell$ are defined, we return to the non-separability encountered in \eqref{xitildeupdate}.
Recall that the subproblem in Step 3 of F-ADMM (Algorithm \ref{GSADMM}) is equivalent to \eqref{FADMM_Lagrangeupdate}, which in turn is equivalent to \eqref{xitildeupdate} + $\frac12 \|\mathbf{x}_i - \mathbf{x}_i\kk\|_{\p_i}^2$.  We concentrate on the regularization term, and notice that
\begin{eqnarray}
  \frac12\mathbf{x}_i^T\p_i \mathbf{x}_i &\overset{\eqref{Groupedx}+\eqref{Groupedindexsets}}{=}&
 \frac12 \begin{bmatrix}
  x_{\mathcal{S}_{i,1}}^T & \dots &x_{\mathcal{S}_{i,p}}^T
\end{bmatrix}
\begin{bmatrix}
    P_{\mathcal{S}_{i,1}}x_{\mathcal{S}_{i,1}}-\rho A_{\mathcal{S}_{i,1}}^T\Big(\sum_{\overset{j =1}{j\neq 1}}^pA_{\mathcal{S}_{i,j}}x_{\mathcal{S}_{i,j}}\Big) \\
    \vdots\\
    P_{\mathcal{S}_{i,p}}x_{\mathcal{S}_{i,p}}-\rho A_{\mathcal{S}_{i,p}}^T
    \Big(\sum_{ \overset{j =1}{j\neq p}}^pA_{\mathcal{S}_{i,j}}x_{\mathcal{S}_{i,j}}\Big) \\
  \end{bmatrix}
  \notag \\
  &=& \frac12\sum_{j=1}^{p} \|x_{\mathcal{S}_{i,j}}\|_{P_{\mathcal{S}_{i,j}}}^2 - \frac{\rho}2 \sum_{j=1}^{p}
  \sum_{ \overset{l =1}{l\neq j}}^p \langle  A_{\mathcal{S}_{i,j}}x_{\mathcal{S}_{i,j}}, A_{\mathcal{S}_{i,l}}x_{\mathcal{S}_{i,l}} \rangle
  \notag \\
  &\overset{\eqref{are-equal}}{=}& \frac12\sum_{j\in\Sscr_i} \|x_j\|_{P_{j}}^2 - \frac{\rho}2 \sum_{j\in\Sscr_i}
  \sum_{\overset{l\in\Sscr_i}{l\neq j}} \langle  A_l x_l, A_j x_j  \rangle. \label{Pinnerprod1}
\end{eqnarray}
Following a similar argument, we can write
\begin{align}
  \mathbf{x}_i^T\p_i \mathbf{x}_i\kk
  &= \sum_{j=1}^{p} x_{\mathcal{S}_{i,j}}^T P_{\Sscr_{i,j}} x_{\mathcal{S}_{i,j}}\kk
  - \rho  \sum_{j=1}^{p} \sum_{\overset{l =1}{l\neq j}}^p \langle  A_{\mathcal{S}_{i,j}}x_{\mathcal{S}_{i,j}}, A_{\mathcal{S}_{i,l}}x_{\mathcal{S}_{i,l}}\kk\rangle
  \nonumber \\
   &= \sum_{j\in\Sscr_i} x_j^T P_j x_j\kk
  - \rho  \sum_{j\in\Sscr_i} \sum_{\overset{l \in\Sscr_i}{l\neq j}}
  \langle
  A_l x_l\kk, A_j x_j
  \rangle.
  \label{Pinnerprod2}
\end{align}
We may now use~\eqref{Pinnerprod1} and \eqref{Pinnerprod2} to write
\begin{multline*}
  \frac12\|\mathbf{x}_i - \mathbf{x}_i\kk\|_{\p_i}^2
  = \frac12\sum_{j\in\Sscr_i} \|x_j\|_{P_{j}}^2
  - \sum_{j\in\Sscr_i} x_j^T P_j x_j\kk
  +\frac12\sum_{j\in\Sscr_i} \|x_j\kk\|_{P_{j}}^2 \\
  + \rho  \sum_{j\in\Sscr_i} \sum_{\overset{l \in\Sscr_i}{l\neq j}}
  \langle
  A_l x_l\kk, A_j x_j
  \rangle
   - \frac{\rho}2 \sum_{j\in\Sscr_i}
  \sum_{\overset{l\in\Sscr_i}{l\neq j}} \langle  A_l x_l, A_j x_j  \rangle
   - \frac{\rho}2 \sum_{j\in\Sscr_i}
  \sum_{\overset{l\in\Sscr_i}{l\neq j}} \langle  A_l x_l\kk, A_j x_j\kk  \rangle.
\end{multline*}
This may be equivalently written as
\begin{multline*}
 \frac12\|\mathbf{x}_i - \mathbf{x}_i\kk\|_{\p_i}^2 \\
 =  \frac12 \sum_{j\in\Sscr_i} \| x_j - x_j\kk \|_{P_j}^2
 + \rho  \sum_{j\in\Sscr_i} \sum_{\overset{l \in\Sscr_i}{l\neq j}}
  \langle
  A_l x_l\kk, A_j x_j
  \rangle
   - \frac{\rho}2 \sum_{j\in\Sscr_i}
  \sum_{\overset{l\in\Sscr_i}{l\neq j}} \langle  A_l x_l, A_j x_j  \rangle
   - \frac{\rho}2 \sum_{j\in\Sscr_i}
  \sum_{\overset{l\in\Sscr_i}{l\neq j}} \langle  A_l x_l\kk, A_j x_j\kk  \rangle.
\end{multline*}
By adding this regularization term, i.e., $\frac12\|\mathbf{x}_i - \mathbf{x}_i\kk\|_{\p_i}^2$, to the objective function in \eqref{xitildeupdate}, we obtain (ignoring terms independent of $\mathbf{x}_i$) the F-ADMM update
\begin{multline*}
\mathbf{x}_i^{(k+1)}
= \arg \min_{\mathbf{x}_i}\Big{\{} \mathbf{f}_i(\mathbf{x}_i) - \langle y^{(k)}\!,\mathcal{A}_i\mathbf{x}_i -\mathbf{b}_i\rangle
  - \rho\langle \mathcal{A}_i\mathbf{x}_i,\mathbf{b}_i\rangle \\
  + \frac{\rho}2\sum_{j\in \mathcal{S}_i} \| A_jx_j\|_2^2
  + \rho
  \sum_{j\in \mathcal{S}_i}
  \sum_{ \overset{l \in \mathcal{S}_i}{l\neq j}} \langle A_l x_l\kk, A_j x_j\rangle
   + \frac12 \sum_{j\in\Sscr_i} \| x_j - x_j\kk \|_{P_j}^2
   \Big\},
\end{multline*}
which is equivalent (again ignoring constant terms) to
\begin{align}
\notag
\mathbf{x}_i^{(k+1)}
&= \arg \min_{\mathbf{x}_i}
 \!\sum_{j\in\Sscr_i}\!
 \Big\{f_j(x_j)
  - \langle y^{(k)}\!, A_j x_j \rangle
  - \rho \langle A_j x_j,\mathbf{b}_i\rangle
  + \frac{\rho}2 \| A_jx_j\|_2^2
  + \rho  \sum_{ \overset{l \in \mathcal{S}_i}{l\neq j}} \langle A_l x_l\kk, A_j x_j\rangle
   + \frac12 \| x_j - x_j\kk \|_{P_j}^2 \!\Big\}\\
   &= \arg \min_{\mathbf{x}_i}
 \!\sum_{j\in\Sscr_i}\!
 \Big\{f_j(x_j) + \frac{\rho}2\|A_j x_j + \sum_{ \overset{l \in \mathcal{S}_i}{l\neq j}} A_l x_l\kk- \mathbf{b}_i - \frac{y\kk}{\rho}\|_2^2
   + \frac12 \| x_j - x_j\kk \|_{P_j}^2 \!\Big\}.
\end{align}
The regularization matrix $\p_i$, defined in \eqref{Pi}, has caused the cross-product term to be eliminated from \eqref{xitildeupdate} (recall that~\eqref{xitildeupdate} was the update \emph{without} using the regularization term), and subsequently the subproblem for updating $\mathbf{x}_i^{(k+1)}$ is separable into $p$ blocks (one solve for each $j\in\Sscr_i$). That is, the decision variables $x_j$ for $j \in \mathcal{S}_i$ can be solved for \emph{in parallel}.
This updating strategy forms our hybrid algorithm H-ADMM, which is able to use a combination of \emph{both} Jacobi and Gauss-Seidel updates. We emphasize that H-ADMM is a special case of Algorithm~\ref{GSADMM}, where the blocks of variables have been (implicitly) grouped together as in \eqref{Problem_lblockform}, and the regularization matrices have the form \eqref{Pi}.

\subsubsection{The H-ADMM Algorithm}

The following is a formal statement of our H-ADMM algorithm. Recall
that H-ADMM is a special case of F-ADMM, and convergence of H-ADMM
follows directly from the convergence theory for F-ADMM.
\begin{algorithm}[H]
\caption{H-ADMM for solving problem~\eqref{Problem}. }\label{HADMM}
  \begin{algorithmic}[1]
    \State \textbf{Initialize:} $x^{(0)} \in \R^N$, $y^{(0)} \in \R^m$, iteration counter $k=0$,
    	       parameters $\rho>0$ and $\gamma \in(0,2)$, data partition index sets $\{\mathcal{S}_i\}_{i = 1}^\ell$,
	       and regularization matrices $\{P_i\}_{i = 1}^n$ satisfying Assumption~\ref{Assume_Pj}.
    \While {stopping condition has not been met}
    \For{$i = 1,\dots,\ell$ in a Gauss-Seidel fashion solve} \label{step-outerfor}
    \State Set $\mathbf{b}_i \gets b - \sum_{q=1}^{i-1}\mathcal{A}_q\mathbf{x}_q\kpo - \sum_{s=i+1}^{\ell}\mathcal{A}_s\mathbf{x}_s\kk$.\label{step-bup}
        \For{$j \in \mathcal{S}_i$ (in parallel)}\label{step-innerfor}
        \begin{eqnarray}\label{HGADMM-parallelloop}
          x_j\kpo \!\!\!&\gets& \!\!\!\arg \min_{x_j}
 \sum_{j\in\Sscr_i}
 \Big\{f_j(x_j) + \frac{\rho}2\|A_j x_j + \sum_{ \overset{l \in \mathcal{S}_i}{l\neq j}} A_l x_l\kk- \mathbf{b}_i - \frac{y\kk}{\rho}\|_2^2
   + \frac12 \| x_j - x_j\kk \|_{P_j}^2 \Big\}
        \end{eqnarray}
        \EndFor
    \EndFor
\State Update the dual variables:
\begin{equation} \label{eq:yup-bd}
     y^{(k+1)} \leftarrow y^{(k)} - \gamma\rho(Ax^{(k+1)}-b ).
\end{equation}
\State Set $k\gets k+1$.
\EndWhile
  \end{algorithmic}
\end{algorithm}

The groups of data are updated in a Gauss-Seidel scheme (see the \textbf{for} loop in Step~\ref{step-outerfor}), while the individual blocks within each group are updated in a Jacobi (parallel) scheme (see the inner \textbf{for} loop in Step~\ref{step-innerfor}).

We have presented H-ADMM as a (serial) Gauss-Seidel algorithm that has
an inner loop which can be executed in parallel, i.e., H-ADMM is
partially parallel. However, we can also view H-ADMM
as a fully parallel method that occasionally inserts updated
information during the update from $\mathbf{x}\kk$ to $\mathbf{x}\kpo$. This shows that
H-ADMM is extremely flexible.

\begin{remark}
  Notice that the regularization matrix \eqref{Pi} is not explicitly formed in H-ADMM (Algorithm~\ref{HADMM}). Specifically, H-ADMM only uses the matrices $P_{\mathcal{S}_{i,1}},\dots,P_{\mathcal{S}_{i,p}}$, which lie on the main (block) diagonal of $\p_i$. Therefore, as for F-ADMM, only one $N_i \times N_i$ regularization matrix $P_i$ is required by H-ADMM for each of the $i = 1,\dots,n$ (individual) blocks.
\end{remark}

\begin{remark} \label{remark-equiv}
  The update \eqref{HGADMM-parallelloop} in H-ADMM has the same form
  as the update \eqref{JADMMupdate} in J-ADMM (Algorithm~\ref{JADMM})
  for all $j \in \mathcal{S}_i$. Therefore, a single iteration of
  H-ADMM has essentially the same cost as that of J-ADMM. However,
  H-ADMM has the advantage of using the most recent updates when
  updating the groups of data.
\end{remark}

\subsection{Computational Considerations}\label{S_practical}

Parallel algorithms are imperative on modern computer architectures, which is why, at face value, Jacobi-type methods seem to have significant advantages over Gauss-Seidel-type competitors. The H-ADMM (Algorithm \ref{HADMM}) bridges the gap between purely Jacobi or purely Gauss-Seidel updates, finding a balance between ensuring algorithm speed via parallelization and allowing up-to-date information to be fed back into the algorithm. In this section we describe how to choose the number of groups $\ell$ and group size $p$ to ``optimize" H-ADMM from a computational perspective. Moreover, we show that H-ADMM is competitive compared with J-ADMM.

Consider a big data application where the number of blocks $n$ is very large. Moreover, suppose we have access to a parallel machine with $p$ processors, where $p < n$ (or even $p \ll n$). Again we will assume that $n = \ell p$, and the $n$ blocks are organized  into $\ell$ groups of $p$ blocks. We stress that \emph{the number of blocks in each group is the same as the number of processors}.

To implement H-ADMM we first initialize $\mathbf{b}_1$. Then, take the
first group of $p$ blocks and send one block to each of the $p$
processors. These $p$ blocks are updated in parallel as in
\eqref{HGADMM-parallelloop} (Step 5 of Algorithm~\ref{HADMM}). Once
these $p$ blocks have been updated, we have the updated group variable
$\mathbf{x}_1\kpo$ consisting of individual blocks
$x_1\kpo,\dots,x_p\kpo$. We then form $\mathbf{b}_2$ as in Step 4 of
Algorithm~\ref{HADMM}. Notice that $\mathbf{b}_2$ incorporates the new
information from the updated block $\mathbf{x}_1\kpo$ via the term
$\mathcal{A}_1\mathbf{x}_1\kpo$, i.e., we feed the updated information back into the algorithm. Now, the next group of $p$ blocks are sent to the $p$ processors to be updated, giving $\mathbf{x}_2\kpo$ consisting of individual blocks $x_{p+1}\kpo,\dots,x_{2p}\kpo$. This new information is then fed back into H-ADMM via the vector $\mathbf{b}_3$. The process is repeated until a full sweep of the data has been completed, i.e., all $n$ blocks have been updated.

In this way, our H-ADMM algorithm has (essentially) the same computational cost as J-ADMM, because the data blocks have been grouped in an intelligent way that takes advantage of the processors available. (For J-ADMM, the data blocks also need to be sent to processors in groups of $p$, it is just that, for J-ADMM, there is no need to update the vector $\mathbf{b}_i$ between the $\ell$ sweeps of the processors.) We note that for J-ADMM, the matrix-vector multiplication $Ax\kpo$ is computed once \emph{all $n$ blocks of $x$ have been updated} (i.e., once $x\kpo$ is available), whereas for H-ADMM, the computation of $Ax\kpo$ has been split and performed in stages with the vectors $\mathcal{A}_i\mathbf{x}_i\kpo$ (for $i = 1,\dots,\ell$) computed after \emph{each group of data has been updated} and the sum taken just before the dual variables are updated. Again, this shows that H-ADMM and J-ADMM have approximately the same computational cost, but H-ADMM has the advantage of new information becoming available to the algorithm, which has the potential for H-ADMM to be more efficient.

\begin{remark}
  Notice that if $n \leq p$, then H-ADMM  is essentially equivalent to J-ADMM (applied to strongly convex functions) if we take $\ell = 1$ and replace Step~\ref{step-bup} with $\mathbf{b}_1 \equiv b$.
\end{remark}

\subsubsection{An efficient implementation of Steps 4--6 in Algorithm \ref{HADMM}}

Algorithm \ref{HADMM} was written to match our presentation in the
text.  However, in practice,
it is computationally advantageous to perform Steps 4--6 in a
different, but equivalent way. To
that end, consider the middle term in the minimization subproblem
\eqref{HGADMM-parallelloop}, and using the definition of
$\mathbf{b}_i$ (Step~4 in Algorithm \ref{HADMM}) we have
\begin{align}\label{Eq_efficientbi}
\notag
  A_j x_j + \sum_{ \overset{l \in \mathcal{S}_i}{l\neq j}} A_l x_l\kk- \mathbf{b}_i - \frac{y\kk}{\rho}
  &= A_j x_j + \sum_{ \overset{l \in \mathcal{S}_i}{l\neq j}} A_l x_l\kk + \sum_{q=1}^{i-1}\mathcal{A}_q\mathbf{x}_q\kpo + \sum_{s=i+1}^{\ell}\mathcal{A}_s\mathbf{x}_s\kk- b - \frac{y\kk}{\rho}\\
  \notag
  &= A_j x_j - A_j x_j\kk+ \sum_{l \in \mathcal{S}_i} A_l x_l\kk + \sum_{q=1}^{i-1}\mathcal{A}_q\mathbf{x}_q\kpo + \sum_{s=i+1}^{\ell}\mathcal{A}_s\mathbf{x}_s\kk- b - \frac{y\kk}{\rho}\\
  \notag
  &= A_j x_j - A_j x_j\kk + \mathcal{A}_i\mathbf{x}_i\kk + \sum_{q=1}^{i-1}\mathcal{A}_q\mathbf{x}_q\kpo + \sum_{s=i+1}^{\ell}\mathcal{A}_s\mathbf{x}_s\kk- b - \frac{y\kk}{\rho}\\
  &= A_j x_j - A_j x_j\kk + \sum_{q=1}^{i-1}\mathcal{A}_q\mathbf{x}_q\kpo + \sum_{s=i}^{\ell}\mathcal{A}_s\mathbf{x}_s\kk- b - \frac{y\kk}{\rho}.
\end{align}
Notice that the last 4 terms in \eqref{Eq_efficientbi} are fixed with
respect to $j \in \mathcal{S}_i$, so we can combine them into a single
vector $v_i$ say, and rewrite Steps 4--6 in Algorithm~\ref{HADMM} as follows.

\begin{algorithm}[H]
\caption{An efficient implementation to replace Steps 4--6 in H-ADMM.}\label{HADMM_efficient}
  \begin{algorithmic}[1]
    \State Set $v_i \gets \sum_{q=1}^{i-1}\mathcal{A}_q\mathbf{x}_q\kpo + \sum_{s=i}^{\ell}\mathcal{A}_s\mathbf{x}_s\kk- b - \frac{y\kk}{\rho}$.\label{step-bup-2}
        \For{$j \in \mathcal{S}_i$ (in parallel)} 
        \begin{eqnarray}\label{HGADMM-parallelloop2}
          x_j\kpo &\gets& \arg \min_{x_j}
 \sum_{j\in\Sscr_i}
 \Big\{f_j(x_j) + \frac{\rho}{2}\|A_j (x_j-x_j\kk) + v_i\|_2^2
   + \frac12 \| x_j - x_j\kk \|_{P_j}^2 \Big\}.
        \end{eqnarray}
        \EndFor
  \end{algorithmic}
\end{algorithm}

\subsubsection{Practical considerations regarding Assumption \ref{Assume_Pj}}

As discussed in Section \ref{SSS_Pi}, choosing the regularization
matrices to have the form \eqref{Pi} for $i=1,\dots,\ell$ in a manner
that satisfies Assumption \ref{Assume_Pj}, ensures that H-ADMM is globally
convergent. However, we have remarked that
an implementation of H-ADMM
only needs the individual (diagonal) block matrices $\{P_j\}_{j=1}^n$.
The purpose of this section is to translate Assumption
\ref{Assume_Pj}, which is an assumption on the group regularization
matrices $\{\p_i\}_{i=1}^\ell$, into a \emph{practical condition} on the
matrices $\{P_j\}_{j=1}^n$.

To this end, recall the definition of $\p_i$ in \eqref{Pi}.
If we define
\begin{equation}
  \p_i^D \eqdef \begin{bmatrix}
    P_{\mathcal{S}_{i,1}} + \rho A_{\mathcal{S}_{i,1}}^TA_{\mathcal{S}_{i,1}} & & \\
    & \ddots & \\
    & & P_{\mathcal{S}_{i,p}} + \rho A_{\mathcal{S}_{i,p}}^TA_{\mathcal{S}_{i,p}}
  \end{bmatrix}, \qquad
\end{equation}
then Assumption~\ref{Assume_Pj} can be written equivalently as
  $\p_i
  \equiv \p_i^D - \rho \mathcal{A}_i^T\mathcal{A}_i
  \succ \frac{\rho^2}{2 \mu} \|\mathcal{A}_D^T\mathcal{A}_\bigtriangleup\|_2^2 I$,
  which holds if and only if
  $\p_i^D
  \succ \rho \mathcal{A}_i^T\mathcal{A}_i
  + \frac{\rho^2}{2 \mu} \|\mathcal{A}_D^T\mathcal{A}_\bigtriangleup\|_2^2 I$.
Using $\rho\|\mathcal{A}_i\|_2^2 I \succeq \rho\mathcal{A}_i^T
\mathcal{A}_i$, a sufficient condition for Assumption~\ref{Assume_Pj} to
hold is that
\begin{equation} \label{suff-cond}
  \p_i^D
  \succ \rho \|\mathcal{A}_i\|_2^2 I
  + \frac{\rho^2}{2 \mu} \|\mathcal{A}_D^T\mathcal{A}_\bigtriangleup\|_2^2 I.
\end{equation}
It then follows from the definition of $\p_i^D$ that~\eqref{suff-cond}
will  hold (equivalently, Assumption~\ref{Assume_Pj} will be
satisfied) if the matrices
$\{P_j\}_{j\in\Sscr_i}$ 
are chosen to satisfy
\begin{equation}\label{PracticalPicondition}
  P_j + \rho A_j^TA_j
  \succ
  \rho \|\mathcal{A}_i\|_2^2 I
  + \frac{\rho^2}{2\mu}\|\mathcal{A}_D^T\mathcal{A}_\bigtriangleup\|_2^2 I
   \qquad \forall \; j\in \mathcal{S}_i.
\end{equation}
That is, if \eqref{PracticalPicondition} is satisfied for all $1 \leq
i \leq \ell$, then H-ADMM is globally convergent.

\subsection{A Special Case for Convex Functions}\label{S_HADMMl2}
In this section, we describe how our hybrid algorithm
is appropriate for convex functions (i.e.,
we do not need strong convexity) in the case of $2$ groups.  In
particular, it is based on Algorithm~\ref{GADMM}, which was first
introduced in~\cite{Deng12} and shown to be globally convergent if the
regularization matrices $P_1$ and $P_2$ in \eqref{GADMM_primal} are
chosen appropriately. In particular, the convergence theory introduced
in \cite{Deng12} holds when $P_1\succ 0$ and $P_2 \succ 0$.

During the remainder of this section, we demonstrate that by choosing the regularization matrix appropriately, Algorithm \ref{GADMM} can be extended to handle the $n$ block case while maintaining all existing convergence theory. This is done by following the hybridization scheme introduced previously in this section.

So, suppose that we have an optimization problem of the form \eqref{Problem}, and that we partition the $n$ blocks into 2 groups, i.e., we have $\ell = 2$ groups  and, for simplicity,
assume that $p = n/2$.\footnote{
From a practical perspective it is sensible to let the first group have a cardinality that is a multiple of the number of processors, and then the second group would contain the remaining blocks.} We can then equivalently write our problem in the form \eqref{Problem_lblockform} where
\begin{subequations}\label{GroupedConvex}
  \begin{eqnarray}
  \mathbf{x}_1 = \begin{bmatrix}
    x_1\\ \vdots\\ x_{n/2}
  \end{bmatrix},\quad
  \mathbf{x}_2 = \begin{bmatrix}
    x_{n/2+1}\\ \vdots\\ x_{n}
  \end{bmatrix},
\end{eqnarray}
\begin{eqnarray}
  \mathbf{f}_1(\mathbf{x}_1) = \sum_{i=1}^{n/2} f_i(x_i), \quad
  \mathbf{f}_2(\mathbf{x}_2) = \!\!\!\sum_{i=n/2+1}^{n} \!\!f_i(x_i),
\end{eqnarray}
and
\begin{eqnarray}
  \mathcal{A}_1 = \begin{bmatrix}
    A_1 & \dots & A_{n/2}
  \end{bmatrix}
  \ \ \text{and} \ \
   \mathcal{A}_2 = \begin{bmatrix}
    A_{n/2+1} & \dots & A_{n}
  \end{bmatrix}.
\end{eqnarray}
\end{subequations}
It is clear that we can apply Algorithm \ref{GADMM} to the grouped data \eqref{GroupedConvex}.
If we now choose the regularization matrices $\p_1$ and $\p_2$ to have the same form as in~\eqref{Pi}, we have
\begin{equation}\label{GADMM_P1}
  \p_1 =
  \begin{bmatrix}
    P_1 & -\rho A_1^TA_2 & \dots & -\rho A_1^TA_{n/2}\\
    -\rho A_2^TA_1 & P_2 & & \vdots\\
    \vdots & & \ddots & \vdots\\
    -\rho A_{n/2}^TA_1 & \dots& \dots& P_{n/2}
  \end{bmatrix}
  \end{equation}
  and
  \begin{equation}\label{GADMM_P2}
  \p_2 =
  \begin{bmatrix}
    P_{n/2+1} & -\rho A_{n/2+1}^TA_{n/2+2} & \dots & -\rho A_{n/2+1}^TA_{n}\\
    -\rho A_{n/2+2}^TA_{n/2+1} & P_{n/2+2} & & \vdots\\
    \vdots & & \ddots & \vdots\\
    -\rho A_{n}^TA_{n/2+1} & \dots& \dots& P_{n}
  \end{bmatrix}.
\end{equation}
Moreover, by letting
\begin{equation*}
  \p_1^D =
  \begin{bmatrix}
    P_1 + \rho A_1^TA_1&   & \\
     &  \ddots & \\
     & & P_{n/2} + \rho A_{n/2}^TA_{n/2}
  \end{bmatrix}
  \ \ \text{and} \ \
  \p_2^D =
  \begin{bmatrix}
    P_{n/2+1} + \rho A_{n/2+1}^TA_{n/2+1}&   & \\
     &  \ddots & \\
     & & P_{n} + \rho A_{n}^TA_{n}
  \end{bmatrix}
\end{equation*}
we have that $\p_1 = \p_1^D - \rho\mathcal{A}_1^T\mathcal{A}_1$ and $\p_2 = \p_2^D - \rho\mathcal{A}_2^T\mathcal{A}_2$. Thus, a sufficient condition for the matrices $\p_1$ and $\p_2$ to be positive definite is that
\begin{eqnarray}\label{GADMMn_condition}
  \p_1^D \succ \rho\|\mathcal{A}_1\|_2^2 I \quad \text{and}\quad
  \p_2^D \succ \rho\|\mathcal{A}_2\|_2^2 I.
\end{eqnarray}
We can then see that a sufficient condition for~\eqref{GADMMn_condition} to hold is to choose  the matrices $\{P_j\}_{j=1}^n$  to satisfy
\begin{equation} \label{suff-2groups}
P_j \succ \rho\|\Ascr_1\|_2^2 I \  \text{for $1 \leq j \leq n/2$,}
\ \ \text{and} \ \
P_j \succ \rho\|\Ascr_2\|_2^2 I \ \ \text{for $n/2+1 \leq j \leq n$.}
\end{equation}
In summary, if the matrices $\{P_j\}_{j=1}^n$ are chosen to
satisfy~\eqref{suff-2groups}, then Algorithm~\ref{GADMMn} is
guaranteed to converge for \emph{convex functions}. We also comment
that $\p_1$ and $\p_2$ in \eqref{GADMM_P1} and \eqref{GADMM_P2} need
not be formed explicitly, since Algorithm~\ref{GADMMn} only requires
the block diagonal regularization matrices $\{P_j\}_{j=1}^n$ be chosen to satisfy~\eqref{suff-2groups}.
\begin{algorithm}[H]
\caption{H-ADMM($\ell=2$) for solving problem~\eqref{Problem} with a convex objective.}\label{GADMMn}
  \begin{algorithmic}[1]
    \State \textbf{Initialization:} $x^{(0)}\in\R^N$, $y^{(0)} \in \R^m$, iteration counter $k=0$, parameters $\rho>0$ and $\gamma\in (0,2)$, and regularization matrices $\{P_j\}_{j=1}^n$.
    \While {the stopping condition has not been met}
    \State Set $v_1 = \sum_{i=1}^2 \Ascr_i \mathbf{x}_i\kk - b - \frac{y\kk}{\rho}$.
\begin{subequations}
\For{$j \in \{1,\dots,n/2\}$ in parallel}
\begin{eqnarray}\label{GADMM_primal1}
     x_j\kpo &\gets& \min_{x_j}\Big{\{}f_j(x_j) + \frac{\rho}2 \| A_j(x_j - x_j\kk) + v_1\|_2^2 + \frac12\|x_j - x_j\kk\|_{P_j}^2 \Big{\}}
\end{eqnarray}
\EndFor
\State Set $v_2 = \Ascr_1 \mathbf{x}_1\kpo  + \Ascr_2 \mathbf{x}_2\kk - b - \frac{y\kk}{\rho}$.
\For{$j \in\{n/2+1,\dots,n\}$ in parallel}
\begin{eqnarray}\label{GADMM_primal2}
     x_j\kpo &\gets& \min_{x_j}\Big{\{}f_j(x_j) + \frac{\rho}2 \| A_j(x_j - x_j\kk) + v_2\|_2^2 + \frac12\|x_j - x_j\kk\|_{P_j}^2 \Big{\}}
\end{eqnarray}
\EndFor
\State Update the dual variables:
\begin{equation*}
     y^{(k+1)} \leftarrow y^{(k)} - \gamma\rho(Ax^{(k+1)}-b).
\end{equation*}
\end{subequations}
\State Set $k\gets k+1$.
\EndWhile
  \end{algorithmic}
\end{algorithm}

\begin{remark}
  An algorithm similar to Algorithm~\ref{GADMMn} is presented in
  \cite{He14}. However, Algorithm~\ref{GADMMn} is more general because
  the only restriction on the matrices $\{P_j\}_{j=1}^n$, are that
  they are ``positive definite enough'', i.e., they satisfy~\eqref{suff-2groups}. On the other hand, the algorithm in \cite{He14} requires the regularization matrices to take the specific form $c_i\rho A_i^TA_i$, where $A_i$ has full rank and $c_i > n/2$ for all $i=1,\dots,n$ (assuming that the individual blocks are partitioned evenly into 2 groups). These latter conditions are more restrictive, and also do not necessarily mean that the subproblems arising within their algorithm are easier to solve. For the example of $l_1$-minimization subject to equality constraints, the regularization matrices $\{P_j\}_{j=1}^n$ in Algorithm \ref{GADMMn}, can be chosen to have the form $P_i = \tau_j I - \rho A_i^TA_i$ for some $\tau_j$, which means that subproblems \eqref{GADMM_primal1} and \eqref{GADMM_primal2} can be solved using soft-thresholding. This is not possible for the algorithm presented in \cite{He14}. For further details, see the numerical experiments in Section \ref{Section_experiments_l1}.
\end{remark}

\section{Numerical Experiments}\label{S_numericalresults}

In this section we present numerical experiments to demonstrate the
computational performance of F-ADMM (Algorithm \ref{GSADMM}) and
H-ADMM (Algorithm \ref{HADMM}), and compare them with J-ADMM
\cite{Deng14}. All numerical experiments were conducted using {\sc Matlab}
on a PC with an Intel i5-3317U, 1.70GHz processor, and 6Gb RAM.

\subsection{$l_2$-Minimization with Linear Constraints}\label{S_l2experiment}

In this numerical experiment, we consider the problem of determining the solution to an underdetermined system of equations with the smallest $2$-norm. Specifically, we aim to solve
\begin{eqnarray}\label{l2experiment}
  \minimize{x\in\R^N} \quad \frac12\|x\|_2^2 \quad \subject \quad  Ax = b.
\end{eqnarray}
We assume that there are $p = 10$ processors, and then divide the data into $\ell = 10$ groups, each group containing $p = 10$ blocks, with each block of size $N_i = 100$, which results in $N = 10^4$ total variables. We also note that the objective function in~\eqref{l2experiment} is (block) separable and can be written as $f(x) = \sum_{i=1}^{n} f_i(x_i)$, with $n = 100$ and $x_i\in N_i$ and $f_i$ is strongly convex with convexity parameter $\mu_i = 1$ for all $1 \leq i \leq 100$.
The constraint matrix $A\in\R^{m\times N}$ with $m = 3\cdot 10^3$ is chosen to be sparse, with approximately 20 nonzeros per row,
where the nonzeros are taken from a Gaussian distribution. To ensure
that $b \in \text{range}(A)$, we randomly generate a vector $z \in
\R^{N}$ with Gaussian entries, and set $b := Az$ so that the constraints in~\eqref{l2experiment} are feasible.


Notice that for problem \eqref{l2experiment}, the subproblem for the $j$th block of $x$ in H-ADMM can be written as
\begin{eqnarray}\label{l1_subprob}
  x_j\kpo &\gets& \arg\min_{x_j}\Big{\{}\frac12\|x_j\|_2^2 + \frac{\rho}2 \| A_j(x_j-x_j\kk) + v_i\|_2^2 + \frac12\|x_j - x_j\kk\|_{P_j}^2 \Big{\}}
\end{eqnarray}
for $j\in\Sscr_i$, where $v_i$ is defined in Step~\ref{step-bup-2} of Algorithm \ref{HADMM_efficient}. (For F-ADMM the subproblems are solved for all $1 \leq j \leq n$.) Notice that the update $x_j\kpo$ can be found by solving the system of equations
\begin{equation}
  (P_j + I + \rho A_j^TA_j)x_j\kpo = (P_j + \rho A_j^TA_j)x_j\kk - \rho A_j^T v_i.
\end{equation}
This linear system motivates us to choose the regularization matrix $P_j$ to be of the form
\begin{equation} \label{def-tauj}
   P_j = \tau_j I - \rho A_j^TA_j
\end{equation}
for some value $\tau_j$, because it may then be combined with~\eqref{l1_subprob} to give the simple and inexpensive update
\begin{equation}\label{l2xupdateP}
  x_j\kpo = \frac{\tau_j}{\tau_j+1} x_j\kk - \frac{\rho}{\tau_j+1} A_j^T v_i.
\end{equation}
For computational reasons, the fraction $\tau_j/(\tau_j+1)$ should be computed \emph{before} multiplication with $x_j\kk$.

For all of the numerical results reported below, we give the number of epochs required by J-ADMM, F-ADMM, and H-ADMM,
and use algorithm parameters $\rho = 0.1$ and $\gamma = 1$. The terminology ``epoch'' refers to one sweep of the data, i.e. that all $n$ blocks of $x$ are updated once. All reported results are averages over one hundred runs. The stopping condition used in all experiments is $\frac12\|Ax-b\|_2^2 \leq 10^{-10}$.

\subsubsection{Results using the values of $\tau$ that satisfy the theory}\label{SSS_l2theory}

In this section, we give the results of our numerical experiments when $\tau_j$ that defines $P_j$ in~\eqref{def-tauj} is chosen as dictated by theory. Specifically, we have
\begin{itemize}
  \item F-ADMM: $\tau_j = \frac{\rho^2}{2\mu} \|A_D^TA_\bigtriangleup\|_2^2 + \rho\|A_j\|_2^2$ for all $1 \leq j \leq n$ (see Assumption \ref{Assume_Pi})
  \item H-ADMM: $\tau_j = \frac{\rho^2}{2\mu} \|\mathcal{A}_D^T \mathcal{A}_\bigtriangleup\|_2^2  + \rho \|\mathcal{A}_i\|_2^2$ for all  $j\in \mathcal{S}_i$ and $1 \leq i \leq \ell$ (see Assumption~\ref{Assume_Pj}  and~\eqref{PracticalPicondition})
  \item J-ADMM: $\tau_j = \rho(n-1)\|A_j\|_2^2$ for all $1 \leq j \leq n$ (see~\cite{Deng14})
\end{itemize}
for the three methods.
%
To get a sense of the size of these choices for $\tau_j$,  we plot their magnitudes in Figure~\ref{Figure_tauvalues}.  The x-axis represents the block number and the y-axis the value of $\tau_j$.  For example, a blue point at the value $(20,180)$ means that $\tau_{20} = 180$.
We can clearly see that the $\tau_j$ values are much smaller for H-ADMM and F-ADMM,  than for J-ADMM. Moreover, the $\tau_j$ values used for F-ADMM and H-ADMM are similar in magnitude. This is, perhaps, a disadvantage since a large value for $\tau_j$ translates into stronger regularization in each subproblem~\eqref{l1_subprob}, which in turn translates into smaller steps and potentially slower convergence.  For our test problem~\eqref{l2experiment}, this turns out to be the case, as we now discuss.
\begin{figure}[H]\centering
  \includegraphics[width = 15cm]{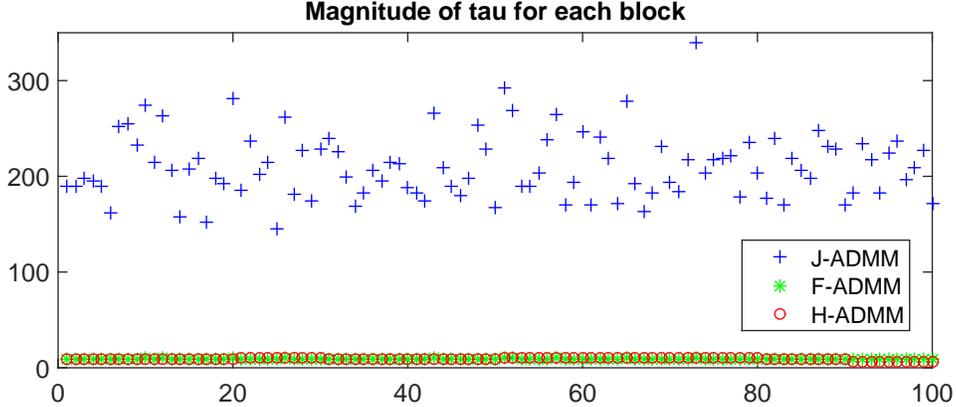}
  \caption{A plot of the magnitude of $\tau_i$ for each block $1 \leq i \leq n$ for problem \eqref{l2experiment}.}
  \label{Figure_tauvalues}
\end{figure}

In Table~\ref{Table_l2theory}, we present the number of epochs needed by each method (averaged over 100 runs).
They show that H-ADMM and F-ADMM require significantly fewer epochs than J-ADMM to determine the solution of problem \eqref{l2experiment} when the theoretical values of $\tau_j$ are chosen.  As discussed in the previous paragraph, we can see that the larger values for $\tau_j$ needed by J-ADMM lead to poor numerical performance compared with F-ADMM and H-ADMM.
However, we remind the reader that H-ADMM and J-ADMM are essentially the same cost per epoch (see Remark~\ref{remark-equiv}), while F-ADMM is generally more costly due to its sequential nature.


\begin{table}[h!]\centering
\begin{tabular}{|c|c|c|}
  \hline
   J-ADMM & H-ADMM & F-ADMM\\
   \hline
       4358.2 & 214.1 & 211.3\\
  \hline
\end{tabular}
\caption{We present the number of epochs required by J-ADMM, F-ADMM, and H-ADMM for the $l_2$-minimization problem \eqref{l2experiment} using theoretical values of $\tau_j$ for $j=1,\dots,n$.}
\label{Table_l2theory}
\end{table}

\subsubsection{Results using values for $\tau$ obtained by parameter tuning}\label{SSS_l2practice}

In~\cite{Deng14}, it was mentioned that J-ADMM displays better
practical performance for smaller values of $\tau_j$ than those
required by the convergence theory. In this section, we compare the
number of epochs required by H-ADMM, F-ADMM, and J-ADMM when $\tau_j$
is allowed to be obtained through parameter tuning. In this
experiment, for simplicity, we assign the same value $\tau_j$ for all
blocks $j=1,\dots,n$. (i.e., $\tau_1 = \tau_2 = \dots = \tau_n$.)
Moreover, we picked the starting value to be $\tau_j =
\frac{\rho^2}{2} \|A\|^4$ because it approximates the values of
$\tau_j$ given by theory, in the sense that:
$\|A_D^TA_{\bigtriangleup}\|^2\leq \|A_D\|^2\|A_{\bigtriangleup} \|^2
\approx \|A\|^2\|A\|^2$.~\footnote{There are many other ways that parameter tuning could be implemented, and we have simply implemented one possibility.}



\newcommand{\strt}{\rule[-.5ex]{0pt}{3ex}}
\newcommand{\hstrt}{\rule[-1ex]{0pt}{3.6ex}}

\begin{table}[h!]\centering
\begin{tabular}{|r|c|c|c|c|c|}
  \hline
  $\tau_j$\phantom{dum} & J-ADMM & H-ADMM & F-ADMM\\
   \hline
  \hstrt$          \frac{\rho^2}{2} \|A\|^4$  & 530.0  & 526.3  & 526.2  \\
  \hstrt$0.6 \cdot \frac{\rho^2}{2} \|A\|^4$  & 324.0  & 320.1  & 319.9  \\
  \hstrt$0.4 \cdot \frac{\rho^2}{2} \|A\|^4$  & 217.7  & 214.5  & 214.1  \\
  \hstrt$0.22\cdot \frac{\rho^2}{2} \|A\|^4$  & 123.1  & 119.3  & 119.0  \\
  \hstrt$0.2 \cdot \frac{\rho^2}{2} \|A\|^4$  & ---    & 95.8   &  95.5  \\
  \hstrt$0.1 \cdot \frac{\rho^2}{2} \|A\|^4$  & ---    & 75.3   &  73.0  \\
  \hline
\end{tabular}
\caption{We present the number of epochs required by J-ADMM, F-ADMM, and H-ADMM for the $\ell_2$-minimization problem~\eqref{l2experiment} for varying values of $\tau_j$. Here, $\tau_j$ takes the same value for all blocks $j=1,\dots,n$.}
\label{l2_practice}
\end{table}
Table~\ref{l2_practice} presents the number of epochs required by J-ADMM, F-ADMM, and H-ADMM on problem \eqref{l2experiment} as $\tau_j$ varies. For each $\tau_j$ we run each algorithm (J-ADMM, F-ADMM and H-ADMM) on 100 random instances of the problem formulation described in Section~\ref{S_l2experiment}.
It is clear that all algorithms require fewer epochs to satisfy the
stopping tolerance as $\tau_j$ decreases. Moreover, for fixed
$\tau_j$, F-ADMM and H-ADMM require slightly fewer epochs than
J-ADMM. Table~\ref{l2_practice} also shows that F-ADMM and H-ADMM will
converge, in practice, for smaller values of $\tau_j$ than J-ADMM. In
particular, J-ADMM diverged when we set $\tau_j = 0.2 \cdot
\frac{\rho^2}{2} \|A\|^4$, whereas F-ADMM and H-ADMM converged for
$\tau_j$ as small as $0.1\cdot \frac{\rho^2}{2} \|A\|^4$; both
diverged for $\tau_j = 0.09\cdot \frac{\rho^2}{2} \|A\|^4$. It is
clear that, when the parameter $\tau_j$ is tuned, F-ADMM and H-ADMM
outperform J-ADMM, when performance is measured in terms of the number
of epochs.

\subsection{$l_1$-Minimization with Linear Constraints}
\label{Section_experiments_l1}
We now consider the problem of $l_1$-minimization subject to equality constraints as given by
\begin{eqnarray}\label{CSexperiment}
  \minimize{x} \quad \|x\|_1 \quad \subject \quad  Ax = b,
\end{eqnarray}
which  arises frequently in the compressed sensing and machine learning literature.
The one norm promotes sparse solutions, while the linear constraints ensure data fidelity. Note that the one norm is separable.

Problem \eqref{CSexperiment} is convex and not strongly convex, which means that H-ADMM($\ell=2$) (Algorithm~\ref{GADMMn}) is guaranteed to converge, while convergence for F-ADMM and H-ADMM has not yet been established.  Nonetheless, we include them in the numerical experiments to study their practical performance.

For ease of comparison, we follow the experiment setup given in \cite{Deng14}. In particular, suppose that the data is partitioned into $n = 100$ blocks of size $N_i = 10$ for all $1 \leq i \leq n$, so that $N = \sum_{i=1}^n N_i = 1000$. We suppose that $A = [A_1,\dots,A_n]$ is randomly generated with Gaussian entries, and that $A_i \in \R^{m \times N_i}$ for each $1 \leq i \leq n$ and $m = 300$, which means that $A \in \R^{m \times N}$.
The sparse signal $x^*$ has $k=60$ randomly located nonzero entries, the nonzero entries are Gaussian, and the vector $b$ is defined by $b := Ax^*$.

For every algorithm we set $\gamma = 1$ and $\rho = 10/\|b\|_1$. 
For H-ADMM we let $n=p\ell$ with $p=4$ and $\ell = 25$, and for H-ADMM($\ell=2$) we set $\ell = 2$ with both groups containing $p=50$ blocks.
All algorithms were terminated when $\|x-x^*\|_2/\|x^*\|_2 \leq10^{-10}$. We report on the number of epochs (as in the previous section) and the final constraint residual $\frac12\|r\|_2^2$, where $r= Ax-b$ for J-ADMM, F-ADMM, H-ADMM, and H-ADMM($\ell=2$). All reported results are averages over 100 runs.

For problem \eqref{CSexperiment},  the subproblem for the $j$th block of $x$ in H-ADMM can be written as
\begin{eqnarray}\label{l1_subprob-2}
  x_j\kpo &\gets& \arg\min_{x_j}\Big{\{}\|x_j\|_1 + \frac{\rho}2 \| A_j(x_j-x_j\kk) + v_i\|_2^2 + \frac12\|x_j - x_j\kk\|_{P_j}^2 \Big{\}},
\end{eqnarray}
for $j\in\Sscr_i$, where $v_i$ is defined in Step~\ref{step-bup-2} of Algorithm \ref{HADMM_efficient}. (For F-ADMM the subproblems are solved for all $1 \leq j \leq n$.)
Next, using a similar choice for $P_j$ as given by~\eqref{def-tauj}, the latter two terms in~\eqref{l1_subprob-2} become
\begin{align*}
  &\phantom{=} \frac{\rho}2 \| A_j(x_j-x_j\kk) + v_i\|_2^2 + \frac12\|x_j - x_j\kk\|_{P_j}^2 \\
  &= \frac{\rho}2  (x_j-x_j\kk)^TA_j^TA_j(x_j-x_j\kk) + \rho(x_j-x_j\kk)^TA_j^T v_i
  + \frac12(x_j-x_j\kk)^TP_j (x_j-x_j\kk) + \frac{\rho}{2}\|v_i\|_2^2 \\
  &= \frac12(x_j-x_j\kk)^T(P_j + \rho A_j^TA_j)(x_j-x_j\kk)+ \rho(x_j-x_j\kk)^TA_j^T v_i + \frac{\rho}{2}\|v_i\|_2^2 \\
  &= \frac{\tau_j}{2}(x_j-x_j\kk)^T (x_j-x_j\kk)+ \rho(x_j-x_j\kk)^TA_j^T v_i + \frac{\rho}{2}\|v_i\|_2^2 \\
  &= \tau_j\Big[ \frac12\|x_j-d_j\|_2^2 - \frac12\| x_j\kk - d_j  \|_2^2 + \frac{\rho}{2\tau_j}\| v_j\|_2^2  \Big],
\end{align*}
where we have defined $d_j := x_j\kk - \frac{\rho}{\tau_j} A_j^T v_i$ to derive the last equality.
Using the previous equality,
the solution to subproblem \eqref{l1_subprob-2} is the same as that given by
\begin{eqnarray}\label{l1_subprob_new}
  x_j\kpo &\gets& \arg\min_{x_j}\Big{\{}\frac{1}{\tau_j}\|x_j\|_1 + \frac12 \|x_j - d_j\|_2^2 \Big{\}},
\end{eqnarray}
which is separable, so that soft thresholding can be used to solve for $x_j\kpo$.

\subsubsection{Results using the values of $\tau$ that satisfy the theory}\label{SSS_CStheory}

Here we present the results of the above stated experiment setup when the values of $\tau_j$ required by the theory are used. We recall that the convergence theory for F-ADMM and H-ADMM has not been established in the convex case, so here we simply use the values of $\tau_j$ that are needed in the strongly convex case. We also recall that convergence of H-ADMM($\ell=2$) \emph{is} guaranteed in the convex case (see Section~\ref{S_HADMMl2}).  Thus, in addition to the $\tau_j$ values for F-ADMM, H-ADMM, and J-ADMM given in Section~\ref{SSS_l2theory}, we also use
\begin{itemize}
   \item H-ADMM($\ell=2$): $\tau_j = \rho \|\mathcal{A}_i\|_2^2$
   for all $j\in \mathcal{S}_i$ and $1 \leq i \leq 2$ (see~\eqref{GADMMn_condition}).
\end{itemize}

Figure \ref{Figure_l1tau} shows typical $\tau_j$ values for each algorithm for the $l_1$-minimization experiment. (For the meaning of each plotted point, see Section~\ref{SSS_l2theory}.) As it was for the $\ell_2$-minimization problem~\eqref{l2experiment}, we again see that the $\tau_j$ values are much smaller for F-ADMM and H-ADMM, when compared to J-ADMM. In addition, the $\tau_j$ values for H-ADMM($\ell=2$) are the smallest overall.
Consequently, the regularization matrices used in H-ADMM($\ell=2$) are significantly less positive definite than all other algorithms, and the regularization matrices for F-ADMM and H-ADMM are less positive definite than for J-ADMM.
\begin{figure}[h!]\centering
  \includegraphics[width = 15cm]{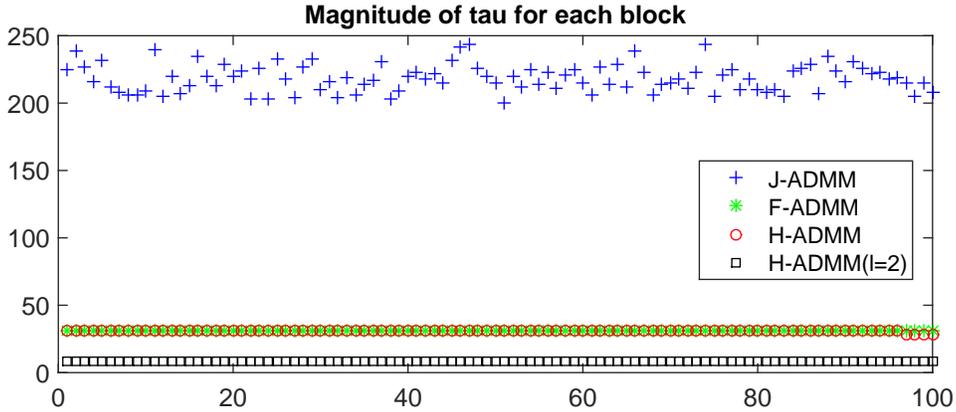}
  \caption{A plot of the magnitude of $\tau_i$ for each block $1 \leq i \leq n$ for problem \eqref{CSexperiment}.}
  \label{Figure_l1tau}
\end{figure}


In Table \ref{CS_theoreticalnoiseless} we give the number of epochs required by each algorithm, as well as the final constraint residual $\frac12\|r\|_2^2$, where $r = Ax-b$. Table \ref{CS_theoreticalnoiseless} shows that H-ADMM($\ell=2$) requires far fewer epochs than the other algorithm, with F-ADMM and H-ADMM requiring about one-sixth the number of epochs compared with J-ADMM, for the $\tau_j$ values stated above. Although there is no convergence theory for F-ADMM and H-ADMM, they both converge in practice for this setup, and are very competitive with J-ADMM.

\begin{table}[h!]\centering
\begin{tabular}{|c|c|c|c|c|c|c|c|c|}
  \hline
  \multicolumn{2}{| c| }{J-ADMM} & \multicolumn{2}{ c| }{H-ADMM($\ell=2$)}& \multicolumn{2}{ c| }{H-ADMM}& \multicolumn{2}{ c| }{F-ADMM} \\
   \hline
   Epochs &  $\frac12\|r\|_2^2$ & Epochs &  $\frac12\|r\|_2^2$ & Epochs &  $\frac12\|r\|_2^2$ & Epochs &  $\frac12\|r\|_2^2$\\
   \hline
     12610.1 & 0.27e-16 & 605.8 & 0.19e-16 & 1882.1 & 0.24e-16 & 1879.4 & 0.24e-16 \\
  \hline
\end{tabular}
\caption{We present the number of epochs required and final constraint violation by J-ADMM, F-ADMM, H-ADMM, and H-ADMM($\ell=2$) for the $l_1$-minimization problem \eqref{CSexperiment} for varying values of $\tau_j$.}
\label{CS_theoreticalnoiseless}
\end{table}

\subsubsection{Results using values for $\tau$ obtained by parameter tuning}

While theory dictates the values of $\tau_j$ needed to guarantee
convergence, experimental performance can often be improved by
selecting better parameter values. In this section, we compare the
performance of the algorithms from the previous section  using smaller
values of $\tau_j$ than those used in Section \ref{SSS_CStheory}. We
repeat the experiments described in 
Section~\ref{Section_experiments_l1}, but now use the same value
$\tau_j$ for all blocks $j=1,\dots,n$ and for all algorithms. The
results are presented in Table~\ref{CS_practicenoiseless}.

Table \ref{CS_practicenoiseless} shows that for the $l_1$-minimization
problem, the number of epochs needed by each of the algorithms to
reach the stopping tolerance decreases as $\tau_j$ decreases. Also,
for each fixed $\tau_j$, J-ADMM requires the most epochs followed by
H-ADMM$(\ell=2)$ and H-ADMM, while F-ADMM requires the smallest number of
epochs. This makes intuitive sense because, during every epoch, F-ADMM
incorporates new information after every block has been updated,
H-ADMM incorporates new information after every group of $p=4$ blocks
have been updated, H-ADMM($\ell=2$) only incorporates new information after half
of the blocks have been updated ($p=n/2 = 50$), while J-ADMM does not
use any updated information within each epoch. 

\begin{table}[H]\centering
\begin{tabular}{|r|c|c|c|c|c|c|c|c|}
  \hline
  & \multicolumn{2}{ c| }{J-ADMM} & \multicolumn{2}{ c| }{H-ADMM($\ell=2$)} & \multicolumn{2}{ c| }{H-ADMM}& \multicolumn{2}{ c| }{F-ADMM}\\
   \hline
   \multicolumn{1}{| c| }{$\tau_j$} & Epochs &  $\frac12\|r\|_2^2$ & Epochs &  $\frac12\|r\|_2^2$ & Epochs &  $\frac12\|r\|_2^2$  & Epochs &  $\frac12\|r\|_2^2$ \\
   \hline
  \hstrt$0.2 \cdot \frac{\rho^2}{2} \|A\|^4$
  & 1078.3 & 0.23e-16 & 1066.2 & 0.23e-16 & 1051.3 & 0.21e-16 & 1053.3 & 0.21e-16\\
  \hstrt$0.1 \cdot \frac{\rho^2}{2} \|A\|^4$
  & 600.3 & 0.19e-16 & 581.5 & 0.19e-16 & 570.0 & 0.20e-16 & 567.8 & 0.22e-16\\
  \hstrt$0.05 \cdot \frac{\rho^2}{2} \|A\|^4$
   & 344.4 & 0.19e-16 & 328.8 & 0.20e-16 & 325.9 & 0.21e-16 & 326.0 & 0.21e-16\\
  \hstrt$0.03 \cdot \frac{\rho^2}{2} \|A\|^4$
   & --- & inf & --- & inf & 246.3 & 0.16e-16 & 244.4 & 0.15e-16\\
  \hstrt$0.02 \cdot \frac{\rho^2}{2} \|A\|^4$
   & --- & inf & --- & inf & 162.2 & 0.34e-16 & 150.3 & 0.29e-16\\
  \hline
\end{tabular}
\caption{We present the number of epochs required and final constraint violation by J-ADMM, F-ADMM, H-ADMM, and H-ADMM($\ell=2$) for the $l_1$-minimization problem~\eqref{CSexperiment} for varying values of $\tau_j$.}
\label{CS_practicenoiseless}
\end{table}


Next, we can also see that J-ADMM and H-ADMM($\ell=2$) perform well until $\tau_j = 0.05 \frac{\rho^2}2 \|A\|_2^4$. However, for smaller values of $\tau_j$, J-ADMM and H-ADMM($\ell=2$) diverged. On the other hand, F-ADMM and H-ADMM still converge (in practice) for $\tau_j = 0.02 \frac{\rho^2}2 \|A\|_2^4$, but diverged when $\tau = 00.19 \frac{\rho^2}2 \|A\|_2^4$. Thus, we can conclude that if the parameter $\tau_j$ is hand-tuned for each algorithm, then practical performance is greatly improved for all methods, and that both F-ADMM and H-ADMM perform the best \emph{in practice} on this convex optimization problem.

\begin{remark}
  An adaptive parameter tuning scheme is presented in
  \cite[Section~2.3]{Deng14}, which ensures that the convergence
  theory developed for J-ADMM still holds, i.e., convergence of J-ADMM
  is guaranteed if their adaptive parameter tuning scheme is
  followed. Unfortunately, we were unable to replicate the numerical
  results presented in that paper, because there was not enough
  information regarding the tuning parameters that they used. However,
  we implemented the adaptive parameter tuning scheme for J-ADMM using
  the following parameters: $\eta = 0.1$, $\alpha_i = 1.1$, $\beta_i =
  0.1$, $Q_i = I$, for all $i=1,\dots,n$, and on average over 100 runs
  on the $l_1$-minimization experiment, J-ADMM required 442.6
  epochs. This is more than the $\approx$ 220 epochs reported in that
  paper. In either case, by hand tuning $\tau_j$, F-ADMM, H-ADMM, and
  H-ADMM($\ell=2$) all outperform J-ADMM.
\end{remark}

\begin{remark}
  Following the same ideas as in \cite[Section~2.3]{Deng14}, it may be
  possible to develop adaptive parameter updating schemes for F-ADMM
  and H-ADMM that still ensure convergence.  In this way, it may be
  possible to achieve additional computational gains for both of them.
\end{remark}

\section{Conclusion}

We presented an algorithm for minimizing block-separable strongly
convex objective functions subject to linear equality constraints.
Our method, called F-ADMM, may be viewed as a flexible version of the
popular ADMM algorithm.  In particular, F-ADMM is provably convergent
for any number of blocks, and contains popular methods such as ADMM,
J-ADMM, and G-ADMM as special cases. 
 
Our work was motived by big data applications.  We showed, via
numerical experiments, that F-ADMM is especially effective when the
number of blocks is larger than the number of available machines.  In
this case, unlike Jacobi methods, our method allows for updated
variables to be used 
when updating the blocks within subsequent groups, all
while maintaining essentially the same cost of a fully Jacobi method.
Our numerical experiments indicate that this approach is more
efficient and stable than the fully
Jacobi method. 


\bibliography{References}
\appendix

\end{document}